\documentclass[12pt,etds]{amsart}
\usepackage{amsmath}
\usepackage{amssymb}
\usepackage{amsfonts}
\usepackage{amsthm}
\usepackage{enumerate}
\usepackage{tikz}
\usepackage{url}
\usetikzlibrary{matrix}

\usepackage{pb-diagram}

\newtheorem{theorem}{Theorem}[section]
\newtheorem{lemma}[theorem]{Lemma}
\newtheorem{prop}[theorem]{Proposition}
\newtheorem{cor}[theorem]{Corollary}
\theoremstyle{remark}

\newtheorem{remark}[theorem]{Remark}

\def\N{{\mathbb N}}

\def\C{{\mathbb C}}

\def\R{{\mathbb R}}
\def\TT{{\mathbb T}}

\def\Z{{\mathbb Z}}

\def\A{{\mathcal{A}}}
\def\B{{\mathcal{B}}}

\def\K{{\mathcal{K}}}

\def\T{{\mathcal{T}}}
\def\I{{\mathcal{I}}}
\def\J{{\mathcal{J}}}
\def\L{{\mathcal{L}}}

\def\M{{\mathcal{M}}}

\newcommand{\clsp}{\overline{\operatorname{span}}}
\newcommand{\lsp}{\operatorname{span}}

\newcommand{\lt}{\operatorname{lt}}
\newcommand{\id}{\operatorname{id}}

\newcommand{\piso}{\operatorname{piso}}
\newcommand{\iso}{\operatorname{iso}}

\newcommand{\End}{\operatorname{End}}
\newcommand{\Aut}{\operatorname{Aut}}

\newcommand{\whitesquare}{\hfill $\whitesquare$\newline\vspace{0.4cm}}

\makeatletter
\@namedef{subjclassname@2020}{\textup{2020} Mathematics Subject Classification}
\makeatother

\numberwithin{equation}{section}

\begin{document}

\title[the partial-isometric crossed product by the semigroup $\N^{2}$]
{The composition series of ideals of the partial-isometric crossed product by the semigroup $\N^{2}$}

\author[Saeid Zahmatkesh]{Saeid Zahmatkesh}
\address{Mathematics and Statistics with Applications (MaSA), Department of Mathematics, Faculty of Science, King Mongkut's University of Technology Thonburi, Bangkok 10140, THAILAND}
\email{saeid.zk09@gmail.com, saeid.kom@kmutt.ac.th}




\begin{abstract}
Suppose that $\alpha$ is an action of the semigroup $\mathbb{N}^{2}$ on a $C^*$-algebra $A$ by endomorphisms. Let
$A\times_{\alpha}^{\textrm{piso}} \mathbb{N}^{2}$ be the associated partial-isometric crossed product. By applying an earlier result
which embeds this semigroup crossed product (as a full corner) in a crossed product by the group $\mathbb{Z}^{2}$, a composition series
$0\leq L_{1}\leq L_{2}\leq A\times_{\alpha}^{\piso} \N^{2}$ of essential ideals is obtained for which we identify the subquotients with
familiar algebras.
\end{abstract}
\maketitle

\section{Introduction}
\label{intro}
It is shown in \cite{SZ2} that the partial-isometric crossed products (Nica-Toeplitz crossed products) by positive cones of abelian lattice-ordered groups are full corners in usual crossed products by groups. This actually generalizes the earlier result in \cite{SZ}, where the case of
abelian totally ordered groups is treated. Now, in the present work, we consider the dynamical system $(A,\N^{2},\alpha)$,
where $\N^{2}$ denotes the positive cone of the abelian lattice-ordered group $\Z^{2}$, and $\alpha$ is an action of $\mathbb{N}^{2}$ on a $C^*$-algebra $A$ by endomorphisms. We would like to recall that we suppose that each endomorphism $\alpha_{t}$ of $A$ extends to
a strictly continuous endomorphism $\overline{\alpha}_{t}$ of the multiplier algebra $\M(A)$ as we deal with non-unital
$C^*$-algebras (see \cite[\S 1]{SZ2} or \cite[\S 1]{SZ3}). So, by \cite[Theorem 4.1]{SZ2}, the partial-isometric crossed product
$A\times_{\alpha}^{\piso} \N^{2}$ of the system $(A,\N^{2},\alpha)$ is (isomorphic to) a full corner of a crossed product by the group
$\Z^{2}$. We apply this corner realization to obtain a composition series of essential ideals
$0\leq L_{1}\leq L_{2}\leq A\times_{\alpha}^{\piso} \N^{2}$ and identify the subquotients with familiar algebras.
In addition, when the action $\alpha$ on $A$ is given by automorphisms, we have simple identifications for the subquotients.
Overall, we think that the present work contains useful information in order to understand the (ideal) structure of the semigroup
crossed product $A\times_{\alpha}^{\piso} \N^{2}$.

Before we proceed, first, readers should be informed that the present work is essentially a revised version
of section 5 of the earlier versions of \cite{SZ2} which are available in the pre-print server arXiv
(see https://arxiv.org/abs/1912.09682v1 and https://arxiv.org/abs/1912.09682v2). Since these versions were too long, a third version
of \cite{SZ2} was then prepared by removing the section 5 and published separately (see https://arxiv.org/abs/1912.09682v3). 
Therefore, the present work is indeed an application of the main theorem in \cite{SZ2} to the system $(A,\N^{2},\alpha)$. We hope
that this clarifies the overlap of the present work with the earlier versions of \cite{SZ2} in arXiv. In addition, to see more on the
theory of partial-isometric crossed products, readers may refer to \cite{Fowler,LR,Adji-Abbas,AZ,AZ2,LZ,SZ3}. In particular, in
\cite{SZ3} partial-isometric crossed products are studied for more general semigroups, namely, (left) LCM semigroups.
However, as a preliminary background for the present work, \cite[\S 2]{SZ2} should be enough for readers to see a quick recall
on partial-isometric and isometric crossed products.

Here is the organization of the present work. It starts with a preliminary section in which the results in \cite{SZ2}
are briefly recalled for the system $(A,\N^{2},\alpha)$. In section \ref{sec:composition}, we show that for the
partial-isometric crossed product $A\times_{\alpha}^{\piso} \N^{2}$ of the system, there is a composition series
$$0\leq L_{1}\leq L_{2}\leq A\times_{\alpha}^{\piso} \N^{2}$$
of essential ideals, and then identify the subquotients with familiar algebras. To do so, we apply the fact that the algebra
$A\times_{\alpha}^{\piso} \N^{2}$ is a full corner in a crossed product by the group $\Z^{2}$ to import the information.
The (essential) ideal $L_{2}$ is the kernel of the natural surjective homomorphism of
$A\times_{\alpha}^{\piso} \N^{2}$ onto the isometric crossed product $A\times_{\alpha}^{\iso} \N^{2}$ of the system.
We show that it is the sum of two essential ideals $\I_{1}$ and $\I_{2}$ (corresponding to two generators of the group $\Z^{2}$), and hence,
the composition series of essential ideals mentioned in above is
$$0\leq \I_{1}\cap \I_{2}\leq \I_{1}+ \I_{2}\leq A\times_{\alpha}^{\piso} \N^{2}.$$
Moreover, while clearly $(A\times_{\alpha}^{\piso} \N^{2})/L_{2}\simeq A\times_{\alpha}^{\iso} \N^{2}$, we show that
the ideal $L_{1}$ is a full corner in the algebra $\K(\ell^{2}(\N^{2}))\otimes A$ of compact operators, and
$L_{2}/L_{1}\simeq \A_{1} \oplus \A_{2}$, where each $\A_{i}$ is a full corner in an algebra of compact operators. Therefore, when
the action $\alpha$ on $A$ is given by automorphisms, we simply have
$(A\times_{\alpha}^{\piso} \N^{2})/L_{2}\simeq A\rtimes_{\alpha} \Z^{2}$, $L_{1}\simeq\K(\ell^{2}(\N^{2}))\otimes A$, and
$L_{2}/L_{1}\simeq \big[\K(\ell^{2}(\N)) \otimes (A\rtimes_{\alpha_{1}} \Z)\big]
\oplus \big[\K(\ell^{2}(\N)) \otimes (A\rtimes_{\alpha_{2}} \Z)\big]$, where $\alpha_{1}$ and $\alpha_{2}$ are two automorphisms
corresponding to two generators of the group $\Z^{2}$.

\section{Preliminaries}
\label{sec:pre}

Let $\N^{2}$ be the positive cone of the abelian lattice-ordered group $\Z^{2}$. Note that, here sometimes an element of
$\Z^{2}$ is simply denoted by $s$ instead of $(s_{1},s_{2})$ for convenience, where each $s_{i}$ belongs to $\Z$. Therefore, $0$
denotes the unit element $(0,0)$ of $\Z^{2}$. Also, we use the additive notation ``$+$" for the action of the group $\Z^{2}$,
and hence, $-s$ denotes the inverse of an element $s$.

Let $(A,\N^{2},\alpha)$ be a dynamical system consisting of a $C^*$-algebra $A$, and an action $\alpha:\N^{2}\rightarrow \End (A)$
of $\N^{2}$ on $A$ by extendible endomorphisms such that $\alpha_{0}=\id$.
Suppose that $(A\times_{\alpha}^{\piso} \N^{2},i_{A},i_{\N^{2}})$ and $(A\times_{\alpha}^{\iso} \N^{2},j_{A},j_{\N^{2}})$ are
the partial-isometric and isometric crossed product of the system, respectively. We recall from \cite{SZ2}
that the pair $(j_{A},j_{\N^{2}})$ induces a surjective homomorphism $q$ of $A\times_{\alpha}^{\piso} \N^{2}$ onto
$A\times_{\alpha}^{\iso} \N^{2}$ such that
$$q(i_{\N^{2}}(m,n)^{*} i_{A}(a) i_{\N^{2}}(s,t))=j_{\N^{2}}(m,n)^{*} j_{A}(a) j_{\N^{2}}(s,t)$$ for all $a\in A$ and $m,n,s,t\in \N$.
Therefore, the following short exact sequence
\begin{align}
\label{exseq1}
0 \longrightarrow \ker q \stackrel{}{\longrightarrow} A\times_{\alpha}^{\piso} \N^{2} \stackrel{q}{\longrightarrow}
A\times_{\alpha}^{\iso} \N^{2} \longrightarrow 0
\end{align}
is obtained, where by \cite[Proposition 2.1]{SZ2}, we have
\begin{align}
\label{ker-q}
\ker q=\clsp\{i_{\N^{2}}(x)^{*}i_{A}(a)(1-i_{\N^{2}}(s)^{*}i_{\N^{2}}(s))i_{\N^{2}}(y): a\in A, x,y,s\in \N^{2}\},
\end{align}
which is an essential ideal of $A\times_{\alpha}^{\piso} \N^{2}$ (see \cite[Proposition 4.3]{SZ2}).
More importantly, the main theorem of \cite{SZ2} says that the algebra $A\times_{\alpha}^{\piso} \N^{2}$ is a full corner in a
crossed product by group. To be more precise, for every $s\in \Z^{2}$, a map $\phi_{s}:A\rightarrow\ell^{\infty}(\Z^{2},A)$ is defined
by
\[
\phi_{s}(a)(t)=
   \begin{cases}
      \alpha_{t-s}(a) &\textrm{if}\empty\ \text{$s\leq t$}\\
      0 &\textrm{otherwise},
   \end{cases}
\]
which is an injective $*$-homomorphism. Note that
\begin{align}
\label{suprem}
\phi_{s}(a)\phi_{t}(b)=\phi_{s\vee t}\big(\alpha_{(s\vee t)-s}(a)\alpha_{(s\vee t)-t}(b)\big)
\end{align}
for all $a,b\in A$ and $s,t\in \Z^{2}$. Now, for the $C^*$-subalgebra $\B$ of $\ell^{\infty}(\Z^{2},A)$ generated by
$\{\phi_{s}(a):s\in \Z^{2}, a\in A\}$, we have
$$\B=\clsp\{\phi_{s}(a):s\in \Z^{2}, a\in A\}.$$
Also, each homomorphism $\phi_{s}:A\rightarrow\mathcal{B}$ extends to a strictly continuous homomorphism
$\overline{\phi}_{s}:\M(A)\rightarrow\M(\B)$ of multiplier algebras (see \cite[Lemma 3.2]{SZ2}), such that
\begin{align}
\label{suprem2}
\overline{\phi}_{s}(m)\overline{\phi}_{t}(n)=
\overline{\phi}_{s\vee t}\big(\overline{\alpha}_{(s\vee t)-s}(m)\overline{\alpha}_{(s\vee t)-t}(n)\big)
\end{align}
for all $s,t\in \Z^{2}$ and $m,n\in \M(A)$. Moreover, \cite[Proposition 3.3]{SZ2} shows that the algebra $\B$ contains an
essential ideal $\J$, such that
\begin{align}
\label{J-span}
\J=\clsp\{\phi_{s}(a)-\phi_{t}(\alpha_{t-s}(a)): s\leq t\in \Z^{2}, a\in A\}.
\end{align}
Next, the shift on $\ell^{\infty}(\Z^{2},A)$ induces an action $\beta$ of $\Z^{2}$ on $\B$ by automorphisms such that
$\beta_{t}\circ\phi_{s}=\phi_{t+s}$ for all $s,t\in \Z^{2}$. So, a group dynamical system $(\B, \Z^{2},\beta)$ is obtained.
If $(\B\rtimes_{\beta} \Z^{2},j_{\B},j_{\Z^{2}})$ is the group crossed product of the system, since $\J$ is a $\beta$-invariant
essential ideal of $\B$, $\J\rtimes_{\beta} \Z^{2}$ sits in $\B\rtimes_{\beta} \Z^{2}$ as an essential ideal (see \cite[Proposition 2.4]{Kusuda}).
Also, recall that, if the maps $\rho:\B\rightarrow \L(\ell^{2}(\Z^{2})\otimes A)$ and $U:\Z^{2} \rightarrow \L(\ell^{2}(\Z^{2})\otimes A)$
are defined by $(\rho(\xi)f)(s)=\xi(s)f(s)$ and $(U_{t}f)(s)=f(s-t)$, respectively, where $\xi\in\B$ and $f\in\ell^{2}(\Z^{2})\otimes A$, then
$\rho$ is a nondegenerate representation and $U$ is a unitary representation such that $\rho(\beta_{t}(\xi))=U_{t}\rho(\xi)U_{t}^{*}$. Therefore,
the pair $(\rho,U)$ is a covariant representation of $(\B,\Z^{2},\beta)$ on $\ell^{2}(\Z^{2})\otimes A\simeq \ell^{2}(\Z^{2},A)$.
Now, if $p=\overline{j_{\B}\circ\phi_{(0,0)}}(1)$, then by \cite[Theorem 4.1]{SZ2}, there is an isomorphism $\Psi$ of
$(A\times_{\alpha}^{\piso} \N^{2},i_{A},i_{\N^{2}})$ onto the full corner $p(\B\rtimes_{\beta} \Z^{2})p$ of the group crossed product
$(\B\rtimes_{\beta} \Z^{2},j_{\B},j_{\Z^{2}})$, such that
\begin{align}
\label{psi-1}
\Psi\big(i_{\N^{2}}(m,n)^{*} i_{A}(a) i_{\N^{2}}(s,t)\big)=pj_{\Z^{2}}(m,n) (j_{\B}\circ\phi_{(0,0)})(a) j_{\Z^{2}}(s,t)^{*}p
\end{align}
for all $m,n,s,t\in\N$ and $a\in A$. Also, by \cite[Lemma 4.2]{SZ2}, the ideal $\ker q$ of $A\times_{\alpha}^{\piso} \N^{2}$
is isomorphic to the full corner $p(\J\rtimes_{\beta} \Z^{2})p$ of the algebra $\J\rtimes_{\beta} \Z^{2}$ via the isomorphism $\Psi$,
such that
\begin{align}
\label{psi-2}
\Psi\big(i_{\N^{2}}(x)^{*}i_{A}(a)(1-i_{\N^{2}}(s)^{*}i_{\N^{2}}(s))i_{\N^{2}}(y)\big)
=p\big[j_{\Z^{2}}(x) j_{\B}\big(\phi_{0}(a)-\phi_{s}(\alpha_{s}(a))\big) j_{\Z^{2}}(y)^{*}\big]p
\end{align}
for all $x,y,s\in \N^{2}$ and $a\in A$.

In addition, see in \cite[\S5]{SZ2} that, if in the system $(A,\N^{2},\alpha)$ the action $\alpha$ is given by automorphisms of $A$,
then we have a simple picture for the algebra $\B$. In this case, the action $\alpha$ extends uniquely to an action
of the group $\Z^{2}$ on $A$ by automorphisms. Therefore, we obtain a group dynamical system
$(B_{\Z^{2}}\otimes A, \Z^{2}, \tau\otimes \alpha^{-1})$, where $B_{\Z^{2}}$ is the $C^{*}$-subalgebra of $\ell^{\infty}(\Z^{2})$
generated by the characteristic functions $\{1_{s}\in\ell^{\infty}(\Z^{2}):s\in \Z^{2}\}$, such that
\[
1_{s}(t)=
   \begin{cases}
      1 &\textrm{if}\empty\ \text{$s\leq t$,}\\
      0 &\textrm{otherwise},
   \end{cases}
\]
and the action $\tau$ of $\Z^{2}$ on $B_{\Z^{2}}$ is given by translation. Let $B_{\Z^{2},\infty}$ be the $C^{*}$-subalgebra
of $B_{\Z^{2}}$ generated by the elements $\{1_{s}-1_{t}: s\leq t\in \Z^{2}\}$, which is actually a $\tau$-invariant
essential ideal of $B_{\Z^{2}}$. Now, the algebra $\B$ is isomorphic to the tensor product $(B_{\Z^{2}}\otimes A)$, where the
isomorphism intertwines the actions $\beta$ and $(\tau\otimes \alpha^{-1})$, and it maps the ideal $\J$ isomorphically onto
the ideal $(B_{\Z^{2},\infty}\otimes A)$ of $(B_{\Z^{2}}\otimes A)$ (see \cite[Proposition 5.2]{SZ2}). As a result, by
\cite[Corollary 5.3]{SZ2}, $A\times_{\alpha}^{\piso} \N^{2}$ and the ideal $\ker q$ are full corners in the group crossed products
$(B_{\Z^{2}}\otimes A)\times_{\tau\otimes\alpha^{-1}} \Z^{2}$ and $(B_{\Z^{2},\infty}\otimes A)\times_{\tau\otimes\alpha^{-1}} \Z^{2}$,
respectively. In particular, for the trivial system $(\C,\N^{2},\id)$, the crossed product $\C\times_{\alpha}^{\piso} \N^{2}$
is isomorphic to the Toeplitz algebra $\T(\Z^{2})$ (see the remark prior to \cite[Lemma 5.4]{SZ2}), and the ideal $\ker q$ is isomorphic to the commutator ideal $\mathcal{C}_{\Z^{2}}$ of $\T(\Z^{2})$ (see \cite[Lemma 5.4]{SZ2}).

\section{The composition series of the ideals of $A\times_{\alpha}^{\piso} \N^{2}$}
\label{sec:composition}
Let $(A,\N^{2},\alpha)$ be a dynamical system consisting of a $C^{*}$-algebra $A$ and an action $\alpha$ of $\N^{2}$ by extendible endomorphisms of $A$. In this section, we show that the crossed product $A\times_{\alpha}^{\piso} \N^{2}$ of the system contains two essential ideals
$\I_{1}$ and $\I_{2}$ corresponding to two generators of the group $\Z^{2}$ such that $\ker q=\I_{1}+\I_{2}$. Therefore, we obtain
a composition series
\begin{align}
\label{compose-1}
0\leq \I_{1}\cap\I_{2}\leq \ker q \leq A\times_{\alpha}^{\piso} \N^{2}
\end{align}
of ideals, for which, we identify the subquotients with familiar algebras. Of course, we already know that
$(A\times_{\alpha}^{\piso} \N^{2}) / \ker q \simeq A\times_{\alpha}^{\iso} \N^{2}$.

To start, firstly, the action $\alpha$ induces two actions $\delta$ and $\gamma$ of $\N$ on $A$ by extendible endomorphisms, such that
\begin{align}
\label{delgam}
\delta_{n}:=\alpha_{(0,n)}\ \textrm{and}\ \gamma_{n}:=\alpha_{(n,0)}
\end{align}
for every $n\in \N$. Hence, two dynamical systems $(A,\N,\delta)$ and $(A,\N,\gamma)$ are obtained, which are actually generated
by the single endomorphisms
$\delta:=\delta_{1}$ and $\gamma:=\gamma_{1}$, respectively. We obviously have
\begin{align}
\label{alph-delgam}
\alpha_{(m,n)}=\delta_{n}\gamma_{m}=\gamma_{m}\delta_{n}
\end{align}
for all $m,n\in \N$. Now, we define two subalgebras $D$ and $C$ of $\ell^{\infty}(\Z,A)$, which are actually the corresponding algebra $\B$
to the totally ordered abelian group $(\Z,\N)$ and the systems $(A,\N,\delta)$ and $(A,\N,\gamma)$, respectively (see also \cite[\S 6]{SZ}).
Thus, we have
$$D=\clsp\{\varphi_{n}(a): n\in \Z, a\in A\}\ \textrm{and}\ C=\clsp\{\psi_{n}(a): n\in \Z, a\in A\},$$ where the maps
$\varphi_{n}:A\rightarrow \ell^{\infty}(\Z,A)$ and $\psi_{n}:A\rightarrow \ell^{\infty}(\Z,A)$ are the (extendible) embeddings defined by
\[
\varphi_{n}(a)(m)=
   \begin{cases}
      \delta_{m-n}(a) &\textrm{if}\empty\ \text{$n\leq m$,}\\
      0 &\textrm{otherwise},
   \end{cases}
\]
and
\[
\psi_{n}(a)(m)=
   \begin{cases}
      \gamma_{m-n}(a) &\textrm{if}\empty\ \text{$n\leq m$,}\\
      0 &\textrm{otherwise}
   \end{cases}
\]
for all $m,n\in \Z$ and $a\in A$, respectively. Note that, the algebras $D$ and $C$ both contain the algebra
$C_{0}(\Z)\otimes A$ as an essential ideal, such that
\begin{eqnarray*}
\begin{array}{rcl}
C_{0}(\Z)\otimes A&=&\clsp\{\varphi_{n}(a)-\varphi_{m}(\delta_{m-n}(a)): n\leq m\in \Z, a\in A\}\\
&=&\clsp\{\psi_{n}(a)-\psi_{m}(\gamma_{m-n}(a)): n\leq m\in \Z, a\in A\}.
\end{array}
\end{eqnarray*}
Indeed, $C_{0}(\Z)\otimes A$ is the corresponding (essential) ideal $\J$ for the algebras $D$ and $C$. Next, for our purpose,
we show that the algebra $\B$ associated with the system $(A,\N^{2},\alpha)$ sits in the algebras
$\ell^{\infty}(\Z,D)$ and $\ell^{\infty}(\Z,C)$ as a $C^{*}$-subalgebra via two families of fibers (column and row fibers).
For every $m,n\in \Z$ and $a\in A$, define a map $\Delta_{(m,n)}: A \rightarrow \ell^{\infty}(\Z,D)$ by
\[
\Delta_{(m,n)}(a)(r)=
   \begin{cases}
      \varphi_{n}(\gamma_{r-m}(a)) &\textrm{if}\empty\ \text{$m\leq r$}\\
      0 &\textrm{otherwise.}
   \end{cases}
\]
Each function $\Delta_{(m,n)}(a)$ can actually be viewed as a sequence of column fibers. Also, each map $\Delta_{(m,n)}$ is a norm-preserving $*$-homomorphism. Now, let $\B_{\delta}$ be the $C^{*}$-subalgebra of $\ell^{\infty}(\Z,D)$ generated by
$\{\Delta_{(m,n)}(a): m,n\in \Z, a\in A\}$. Since calculation shows that $\Delta_{(m,n)}(a)^{*}=\Delta_{(m,n)}(a^{*})$ and $\Delta_{(m,n)}(a)\Delta_{(t,u)}(b)=\Delta_{(x,y)}(c)$, where $(x,y)=(m,n) \vee (t,u)=(m\vee t,n\vee u)$ and $c=\delta_{y-n}(\gamma_{x-m}(a))\delta_{y-u}(\gamma_{x-t}(b))$, it follows that
$$\B_{\delta}=\clsp\{\Delta_{(m,n)}(a): m,n\in \Z, a\in A\}.$$

Similarly, for every $m,n\in \Z$ and $a\in A$, we define a map $\Gamma_{(m,n)}: A \rightarrow \ell^{\infty}(\Z,C)$ by
\[
\Gamma_{(m,n)}(a)(s)=
   \begin{cases}
      \psi_{m}(\delta_{s-n}(a)) &\textrm{if}\empty\ \text{$n\leq s$}\\
      0 &\textrm{otherwise,}
   \end{cases}
\]
which is an injective $*$-homomorphism. In this case, each function $\Gamma_{(m,n)}(a)$ can be viewed as a (columnar) sequence of row fibers.
Then, for the $C^{*}$-subalgebra $\B_{\gamma}$ of $\ell^{\infty}(\Z,C)$ generated by $\{\Gamma_{(m,n)}(a): m,n\in \Z, a\in A\}$,
we have
$$\B_{\gamma}=\clsp\{\Gamma_{(m,n)}(a): m,n\in \Z, a\in A\}.$$

\begin{lemma}
\label{extension}
Each homomorphism $\Delta_{(m,n)}: A \rightarrow \B_{\delta}$ extends to a strictly continuous homomorphism
$\overline{\Delta}_{(m,n)}: \M(A) \rightarrow \M(\B_{\delta})$ of multiplier algebras as well as each homomorphism
$\Gamma_{(m,n)}: A \rightarrow \B_{\gamma}$.
\end{lemma}

\begin{proof}
We skip the proof as it is similar to the proof of \cite[Lemma 3.2]{SZ2}. In brief, this is due to the extendibility
of the endomorphisms $\delta_{n}$ and $\gamma_{n}$, and homomorphisms $\varphi_{n}$ and $\psi_{n}$.
\end{proof}

It therefore follows by Lemma \ref{extension} that
\[
\overline{\Delta}_{(m,n)}(c)(r)=
   \begin{cases}
      \overline{\varphi}_{n}(\overline{\gamma}_{r-m}(c)) &\textrm{if}\empty\ \text{$m\leq r$}\\
      0 &\textrm{otherwise,}
   \end{cases}
\]
for all $m,n\in \Z$ and $c\in \M(A)$ (similarly for $\overline{\Gamma}_{(m,n)}(c)$).

\begin{lemma}
\label{B=Bdelta=Bgamma}
There is an isomorphism $\Lambda_{\delta}$ of $\B$ onto $\B_{\delta}$ such that
$$\Lambda_{\delta}(\phi_{(m,n)}(a))=\Delta_{(m,n)}(a)$$ for all $m,n\in \Z$ and $a\in A$.
Similarly, the algebra $\B$ is isomorphic to the algebra $\B_{\gamma}$ via an isomorphism $\Lambda_{\gamma}$ such that $$\Lambda_{\gamma}(\phi_{(m,n)}(a))=\Gamma_{(m,n)}(a)$$ for all $m,n\in \Z$ and $a\in A$.
\end{lemma}

\begin{proof}
We only prove the existence of the isomorphism $\Lambda_{\delta}$ as the existence of the isomorphism $\Lambda_{\gamma}$ follows similarly.
Define a map
$$\Lambda_{\delta}:\lsp\{\phi_{(m,n)}(a): (m,n)\in\Z^{2}, a\in A\}\rightarrow \B_{\delta}$$ by
\begin{align}
\label{lambda-del}
\Lambda_{\delta}\bigg(\sum_{i} \phi_{(m_{i},n_{i})}(a_{i})\bigg)=\sum_{i} \Delta_{(m_{i},n_{i})}(a_{i}).
\end{align}
Obviously, $\Lambda_{\delta}$ is linear. We show that it preserves the norm, from which, it follows that it is a well-defined linear isometry.
Firstly, for any Hilbert space $H$, there is an isomorphism $U$ of the Hilbert space
$$\ell^{2}(\Z^{2},H)\simeq \ell^{2}(\Z^{2})\otimes H\simeq (\ell^{2}(\Z)\otimes \ell^{2}(\Z))\otimes H$$ onto
the Hilbert space $$\ell^{2}(\Z)\otimes (\ell^{2}(\Z)\otimes H)\simeq \ell^{2}(\Z)\otimes \ell^{2}(\Z,H)$$ which induces
the following isomorphism
\begin{align}
\label{iso-U}
T\in B(\ell^{2}(\Z^{2})\otimes H) \mapsto UTU^{-1}\in B(\ell^{2}(\Z)\otimes \ell^{2}(\Z,H))
\end{align}
of $C^{*}$-algebras. Now, let $\pi:A\rightarrow B(H)$ be a faithful and nondegenerate representation of $A$ on a Hilbert space $H$.
Then, the map $\tilde{\pi}:\B \rightarrow B(\ell^{2}(\Z^{2})\otimes H)$ defined by
$$(\tilde{\pi}(\xi)f)(r,s)=\pi(\xi(r,s))f(r,s)\ \ \ \textrm{for all}\ \xi\in \B\ \textrm{and}\ f\in \ell^{2}(\Z^{2})\otimes H$$
is a nondegenerate and faithful representation of $\B$ on the Hilbert space $\ell^{2}(\Z^{2})\otimes H$.
On the other hand, let $\rho:D\rightarrow B(\ell^{2}(\Z,H))$ be the nondegenerate and faithful representation defined by
$$(\rho(\xi)f)(s)=\pi(\xi(s))f(s)\ \ \ \textrm{for all}\ \xi\in D\ \textrm{and}\ f\in \ell^{2}(\Z,H).$$
Then, $\rho$, itself, induces a map $\tilde{\rho}:\B_{\delta} \rightarrow B(\ell^{2}(\Z)\otimes \ell^{2}(\Z,H))$ defined by
$$(\tilde{\rho}(\eta)g)(r)=\rho(\eta(r))g(r)$$
for all $\eta\in \B_{\delta}$ and $g\in \ell^{2}(\Z)\otimes \ell^{2}(\Z,H)$. It is not difficult to see that $\tilde{\rho}$
is indeed a nondegenerate and faithful representation of the algebra $\B_{\delta}$ on the Hilbert space
$\ell^{2}(\Z)\otimes \ell^{2}(\Z,H)$. Now, calculation on spanning elements shows that
\begin{align}\label{eq17}
U \tilde{\pi}\bigg(\sum_{i} \phi_{(m_{i},n_{i})}(a_{i})\bigg)=\tilde{\rho}\bigg(\sum_{i} \Delta_{(m_{i},n_{i})}(a_{i})\bigg)U,
\end{align}
from which, if follows that
\begin{eqnarray*}
\begin{array}{rcl}
\bigg\| \Lambda_{\delta}\bigg(\displaystyle\sum_{i} \phi_{(m_{i},n_{i})}(a_{i})\bigg) \bigg\|
&=&\bigg\| \displaystyle\sum_{i} \Delta_{(m_{i},n_{i})}(a_{i}) \bigg\|\\
&=&\bigg\| \tilde{\rho}\bigg(\sum_{i} \Delta_{(m_{i},n_{i})}(a_{i})\bigg) \bigg\|\\
&=&\bigg\| U \tilde{\pi}\bigg(\sum_{i} \phi_{(m_{i},n_{i})}(a_{i})\bigg) U^{-1} \bigg\|\\
&=&\bigg\| \tilde{\pi}\bigg(\sum_{i} \phi_{(m_{i},n_{i})}(a_{i})\bigg) \bigg\|
=\bigg\| \displaystyle\sum_{i} \phi_{(m_{i},n_{i})}(a_{i}) \bigg\|.
\end{array}
\end{eqnarray*}
So, $\Lambda_{\delta}$ is a well-defined liner map which preserves the norm, and therefore, it extends to a linear isometry
of $\B$ into $\B_{\delta}$. We use the same notation $\Lambda_{\delta}$ for the extension, which is clearly onto
by (\ref{lambda-del}).

Finally, one can easily calculate on spanning elements to see that $\Lambda_{\delta}$ preserves involution and multiplication, too.
Thus, $\Lambda_{\delta}$ is an isomorphism of $\B$ onto $\B_{\delta}$.
\end{proof}

\begin{prop}
\label{C0 in Bgam-Bdel}
The algebras $\B_{\delta}$ and $\B_{\gamma}$ contain the algebras $C_{0}(\Z) \otimes D$ and $C_{0}(\Z) \otimes C$, respectively, as essential ideals, such that
\begin{align}\label{C0-delta}
C_{0}(\Z) \otimes D=\clsp\{\Delta_{(m,n)}(a)-\Delta_{(t,n)}(\gamma_{t-m}(a)): m,t,n\in \Z\ \textrm{with}\ m\leq t, a\in A\},
\end{align}
and
\begin{align}\label{C0-gamma}
C_{0}(\Z) \otimes C=\clsp\{\Gamma_{(m,n)}(a)-\Gamma_{(m,u)}(\delta_{u-n}(a)): m,n,u\in \Z\ \textrm{with}\ n\leq u, a\in A\}.
\end{align}

\end{prop}

\begin{proof}
We only prove for $C_{0}(\Z) \otimes D$ and skip the proof on $C_{0}(\Z) \otimes C$ as it follows by a similar discussion. Firstly,
the right hand side of (\ref{C0-delta}) is in fact equal to
\begin{align}\label{C0-delta2}
\clsp\{\Delta_{(m,n)}(a)-\Delta_{(m+1,n)}(\gamma(a)): m,n\in \Z, a\in A\}.
\end{align}
This is due to the following calculation:
\begin{eqnarray}\label{eq18}
\begin{array}{l}
\Delta_{(m,n)}(a)-\Delta_{(t,n)}(\gamma_{t-m}(a))\\
=\big[\Delta_{(m,n)}(a)-\Delta_{(m+1,n)}(\gamma(a))\big]
+\big[\Delta_{(m+1,n)}(\gamma(a))-\Delta_{(m+2,n)}(\gamma_{2}(a))\big]\\
+\cdot\cdot\cdot+\big[\Delta_{(t-1,n)}(\gamma_{t-m-1}(a))-\Delta_{(t,n)}(\gamma_{t-m}(a))\big]\\
=\displaystyle\sum_{r=1}^{t-m}\big[\Delta_{(m+r-1,n)}(\gamma_{r-1}(a))-\Delta_{(m+r,n)}(\gamma_{r}(a))\big].
\end{array}
\end{eqnarray}
Thus, we only need to show that
\begin{align}\label{C0-delta3}
C_{0}(\Z) \otimes D=\clsp\{\Delta_{(m,n)}(a)-\Delta_{(m+1,n)}(\gamma(a)): m,n\in \Z, a\in A\}.
\end{align}
Since
$$\Delta_{(m,n)}(a)-\Delta_{(m+1,n)}(\gamma(a))=(...,0,0,0,\varphi_{n}(a),0,0,0,...),$$ where $\varphi_{n}(a)$ is in
the $m$th slot, and elements $(...,0,0,0,\varphi_{n}(a),0,0,0,...)$ span the algebra $C_{0}(\Z) \otimes D$, it follows that
(\ref{C0-delta3}) holds.

Next, to show that $C_{0}(\Z) \otimes D$ is an ideal of $\B_{\delta}$, it is enough to calculate the product
$$\Delta_{(r,s)}(b)\big[\Delta_{(m,n)}(a)-\Delta_{(m+1,n)}(\gamma(a))\big]$$ on spanning elements to see that
it belongs to $C_{0}(\Z) \otimes D$. However, we skip the calculation as it is similar to the one in the proof of
\cite[Proposition 3.3]{SZ2}. To see that the ideal $C_{0}(\Z) \otimes D$ of $\B_{\delta}$ is essential, note that it follows easily from the
fact that $\B_{\delta}\subset \ell^{\infty}(\Z,D)= \M(C_{0}(\Z) \otimes D)$.
\end{proof}

\begin{theorem}
\label{J for Z2}
Let
$$\J_{\delta}:=\clsp\{\phi_{(m,n)}(a)-\phi_{(t,n)}(\gamma_{t-m}(a)): m,t,n\in \Z\ \textrm{with}\ m\leq t, a\in A\},$$
and
$$\J_{\gamma}:=\clsp\{\phi_{(m,n)}(a)-\phi_{(m,u)}(\delta_{u-n}(a)): m,n,u\in \Z\ \textrm{with}\ n\leq u, a\in A\}.$$
Then, $\J_{\delta}$ and $\J_{\gamma}$ are essential ideals of $\B$ such that $\J_{\delta}+\J_{\gamma}=\J$ and
$\J_{\delta}\cap \J_{\gamma}=C_{0}(\Z^{2})\otimes A\simeq C_{0}(\Z)\otimes C_{0}(\Z) \otimes A$, which is an essential ideal of $\B$.
Moreover, $\J_{\delta}$ and $\J_{\gamma}$ are isomorphic to the algebras $C_{0}(\Z)\otimes D$
and $C_{0}(\Z)\otimes C$, respectively.

\end{theorem}

\begin{proof}
By applying the isomorphism $\Lambda_{\delta}:\B\rightarrow \B_{\delta}$ in Lemma \ref{B=Bdelta=Bgamma}, since $C_{0}(\Z)\otimes D$ is
an essential ideal of $\B_{\delta}$ by Proposition \ref{C0 in Bgam-Bdel},
$\J_{\delta}:=\Lambda_{\delta}^{-1}\big(C_{0}(\Z) \otimes D\big)$ is an essential ideal of $\B$ which is clearly isomorphic to
$C_{0}(\Z) \otimes D$. Moreover, the following equation
$$\Lambda_{\delta}^{-1}\big(\Delta_{(m,n)}(a)-\Delta_{(t,n)}(\gamma_{t-m}(a))\big)=\phi_{(m,n)}(a)-\phi_{(t,n)}(\gamma_{t-m}(a))$$
along with (\ref{C0-delta3}) implies that
\begin{eqnarray*}
\begin{array}{rcl}
\J_{\delta}&=&\clsp\{ \phi_{(m,n)}(a)-\phi_{(t,n)}(\gamma_{t-m}(a)): m,t,n\in \Z\ \textrm{with}\ m\leq t, a\in A \}\\
&=&\clsp\{ \phi_{(m,n)}(a)-\phi_{(m+1,n)}(\gamma(a)): m,n\in \Z, a\in A \}.
\end{array}
\end{eqnarray*}
The proof on $\J_{\gamma}$ follows similarly (by using the isomorphism $\Lambda_{\gamma}$) that we skip it here.

Next, we show that $\J=\J_{\delta}+\J_{\gamma}$. The inclusion $\J_{\delta}+\J_{\gamma}\subset \J$ is immediate.
For the other inclusion, take any spanning element $\phi_{(m,n)}(a)-\phi_{(t,u)}(\alpha_{(t-m,u-n)}(a))$ of $\J$, where $a\in A$ and $(m,n), (t,u)\in \Z^{2}$ with $(m,n)\leq (t,u)$. We have
\begin{eqnarray*}
\begin{array}{l}
\phi_{(m,n)}(a)-\phi_{(t,u)}(\alpha_{(t-m,u-n)}(a))\\
=[\phi_{(m,n)}(a)-\phi_{(t,n)}(\gamma_{t-m}(a))]+[\phi_{(t,n)}(\gamma_{t-m}(a))-\phi_{(t,u)}(\delta_{u-n}(\gamma_{t-m}(a)))]\\
=[\phi_{(m,n)}(a)-\phi_{(t,n)}(\gamma_{t-m}(a))]+[\phi_{(t,n)}(b)-\phi_{(t,u)}(\delta_{u-n}(b))]\in (\J_{\delta}+\J_{\gamma}),
\end{array}
\end{eqnarray*}
where $b=\gamma_{t-m}(a)\in A$. This implies that $\J\subset \J_{\delta}+\J_{\gamma}$. Therefore,
$\J=\J_{\delta}+\J_{\gamma}$, from which, it follows that the ideal $\J$ is actually spanned by the elements of the form
\begin{align}\label{J-span-elm}
[\phi_{(m,n)}(a)-\phi_{(m+1,n)}(\gamma(a))]+[\phi_{(t,u)}(b)-\phi_{(t,u+1)}(\delta(b))],
\end{align}
where $m,n,t,u\in \Z$ and $a,b\in A$.

Now, to see that $\J_{\delta}\cap \J_{\gamma}=C_{0}(\Z^{2})\otimes A$, first recall that $C_{0}(\Z^{2})\otimes A$ is spanned
by the elements (finitely supported functions) $f_{(m,n)}^{a}:\Z^{2}\rightarrow A$ defined by
\begin{align}\label{span-el-C0}
f_{(m,n)}^{a}(r,s)=
   \begin{cases}
      a &\textrm{if}\empty\ \text{$(r,s)=(m,n)$}\\
      0 &\textrm{otherwise}
   \end{cases}
\end{align}
for every $a\in A$ and $(m,n)\in \Z^{2}$. Then, it is not difficult to see that each function $f_{(m,n)}^{a}$ is actually equal
to the element
\begin{align}\label{C0-in-Jd}
[\phi_{(m,n)}(a)-\phi_{(m+1,n)}(\gamma(a))]-[\phi_{(m,n+1)}(\delta(a))-\phi_{(m+1,n+1)}(\alpha_{(1,1)}(a))]
\end{align}
of $\J_{\delta}$, which is also equal to
\begin{align}\label{C0-in-Jg}
[\phi_{(m,n)}(a)-\phi_{(m,n+1)}(\delta(a))]-[\phi_{(m+1,n)}(\gamma(a))-\phi_{(m+1,n+1)}(\alpha_{(1,1)}(a))]\in \J_{\gamma},
\end{align}
where $\alpha_{(1,1)}(a)=\gamma(\delta(a))=\delta(\gamma(a))$. This implies that
$C_{0}(\Z^{2})\otimes A\subset \J_{\delta}\cap \J_{\gamma}$. For the other inclusion, as
$\J_{\delta}\cap \J_{\gamma}=\J_{\delta} \J_{\gamma}$, it is enough to show that each product
\begin{align}\label{eq19}
[\phi_{(m,n)}(a)-\phi_{(m+1,n)}(\gamma(a))][\phi_{(t,u)}(b)-\phi_{(t,u+1)}(\delta(b))]
\end{align}
of the spanning elements of $\J_{\delta}$ and $\J_{\gamma}$ belongs to $C_{0}(\Z^{2})\otimes A$. To calculate the product (\ref{eq19}) (of two functions in $\ell^{\infty}(\Z^{2},A)$), think of the intersection point of two discrete rays in $\R^{2}$.
One is vertical corresponding to $[\phi_{(m,n)}(a)-\phi_{(m+1,n)}(\gamma(a))]$ with the initial point $(m,n)$, and the other one is horizontal corresponding to $[\phi_{(t,u)}(b)-\phi_{(t,u+1)}(\delta(b))]$ with
the initial point $(t,u)$. These two rays have only one intersection at the point $(m,u)$ if $m\geq t$ and $n\leq u$. Otherwise, there is no intersection point. This is equivalent to saying that the product (\ref{eq19}), as a function
in $\ell^{\infty}(\Z^{2},A)$, is nonzero only when $m\geq t$ and $n\leq u$. So, in this case, we have
\begin{eqnarray*}
\begin{array}{l}
[\phi_{(m,n)}(a)-\phi_{(m+1,n)}(\gamma(a))][\phi_{(t,u)}(b)-\phi_{(t,u+1)}(\delta(b))]\\
=\phi_{(m,n)}(a)\phi_{(t,u)}(b)-\phi_{(m,n)}(a)\phi_{(t,u+1)}(\delta(b))-\phi_{(m+1,n)}(\gamma(a))\phi_{(t,u)}(b)\\
+\phi_{(m+1,n)}(\gamma(a))\phi_{(t,u+1)}(\delta(b))\\
=\phi_{(m,u)}\big(\delta_{u-n}(a)\gamma_{m-t}(b)\big)-\phi_{(m,u+1)}\big(\delta_{u-n+1}(a)\gamma_{m-t}(\delta(b))\big)\\
-\phi_{(m+1,u)}\big(\delta_{u-n}(\gamma(a))\gamma_{m-t+1}(b)\big)
+\phi_{(m+1,u+1)}\big(\delta_{u-n+1}(\gamma(a))\gamma_{m-t+1}(\delta(b))\big)\ \ \ [\textrm{by}\ (\ref{suprem})]\\
=\big[\phi_{(m,u)}\big(\delta_{u-n}(a)\gamma_{m-t}(b)\big)-\phi_{(m,u+1)}\big(\delta_{u-n+1}(a)\delta(\gamma_{m-t}(b))\big)\big]\\
-\big[\phi_{(m+1,u)}\big(\gamma(\delta_{u-n}(a))\gamma_{m-t+1}(b)\big)
-\phi_{(m+1,u+1)}\big(\gamma(\delta_{u-n+1}(a))\gamma_{m-t+1}(\delta(b))\big)\big]\\
=\big[\phi_{(m,u)}(c)-\phi_{(m,u+1)}(\delta(c))\big]-\big[\phi_{(m+1,u)}(\gamma(c))-\phi_{(m+1,u+1)}(\delta(\gamma(c)))\big],\\
\end{array}
\end{eqnarray*}
where $c=\delta_{u-n}(a)\gamma_{m-t}(b)\in A$. Therefore,
\begin{eqnarray}
\label{span-Jd.Jg}
\begin{array}{l}
[\phi_{(m,n)}(a)-\phi_{(m+1,n)}(\gamma(a))][\phi_{(t,u)}(b)-\phi_{(t,u+1)}(\delta(b))]\\
=\big[\phi_{(m,u)}(c)-\phi_{(m,u+1)}(\delta(c))\big]-\big[\phi_{(m+1,u)}(\gamma(c))-\phi_{(m+1,u+1)}(\delta(\gamma(c)))\big],
\end{array}
\end{eqnarray}
which is equal to the spanning element $f_{(m,u)}^{c}$ of $C_{0}(\Z^{2})\otimes A$ (see (\ref{span-el-C0})-(\ref{C0-in-Jg})). So,
the product (\ref{eq19}) belongs to $C_{0}(\Z^{2})\otimes A$, and hence, $\J_{\delta}\cap \J_{\gamma}\subset C_{0}(\Z^{2})\otimes A$.
At last, since the ideals $\J_{\delta}$ and $\J_{\gamma}$ are both essential, it follows that
$\J_{\delta}\cap \J_{\gamma}=C_{0}(\Z^{2})\otimes A$ must be an essential ideal of $\B$.

\end{proof}

The following are two remarks that are required for Theorem \ref{J crossed Z2}.
\begin{remark}
\label{tensor CP}
Suppose that $(A \rtimes_{\alpha} \Gamma,i_{A},i_{\Gamma})$ and $(B\rtimes_{\beta} G,i_{B},i_{G})$ are the group crossed products of the
dynamical systems $(A,\Gamma,\alpha)$ and $(B,G,\beta)$ by discrete groups, respectively. Recall that there is an action
$\alpha \otimes \beta$ of the group $\Gamma \times G$ on the maximal tensor product $A \otimes_{\max} B$ by automorphisms such that
$(\alpha \otimes \beta)_{(s,t)}=\alpha_{s} \otimes_{\max} \beta_{t}$ for every $(s,t)\in \Gamma\times G$. Then, the corresponding
group crossed product $(A \otimes_{\max} B)\rtimes_{\alpha \otimes \beta} (\Gamma \times G)$ can be decomposed as the maximal tensor
product of the crossed products $A \times_{\alpha} \Gamma$ and $B\times_{\beta} G$. More precisely, there is an isomorphism
$$\Pi:((A \otimes_{\max} B)\rtimes_{\alpha \otimes \beta} (\Gamma \times G),i)\rightarrow
(A \rtimes_{\alpha} \Gamma) \otimes_{\max} (B\rtimes_{\beta} G)$$
such that
$$\Pi\big(i_{(A\otimes_{\max} B)}(a \otimes b)i_{\Gamma\times G}(s,t)\big)=i_{A}(a)i_{\Gamma}(s)\otimes i_{B}(b)i_{G}(t)$$
for all $a\in A$, $b\in B$, and $(s,t)\in \Gamma\times G$.
Now, if in particular $\Gamma=G$, then one can see that the map $\gamma:G \times G\rightarrow \Aut(A \otimes_{\max} B)$
defined by
$$\gamma_{(s,t)}=(\alpha \otimes \beta)_{(t,s)}=\alpha_{t} \otimes_{\max} \beta_{s}\ \ \textrm{for all}\ (s,t)\in G \times G$$
is an action of the group $G \times G$ on the algebra $A \otimes_{\max} B$ by automorphisms. If
$((A \otimes_{\max} B)\rtimes_{\gamma} (G \times G),k)$ is the group crossed product of the system $(A \otimes_{\max} B, G \times G, \gamma)$,
then it is not difficult to see that it is isomorphic to the crossed product
$((A \otimes_{\max} B)\rtimes_{\alpha \otimes \beta} (G \times G),i)$ via an isomorphism $\Pi_{2}$,
such that
$$\Pi_{2}\big(k_{(A\otimes_{\max} B)}(a \otimes b)k_{G \times G}(s,t)\big)
=i_{(A\otimes_{\max} B)}(a \otimes b)i_{G \times G}(t,s)$$ for all $a\in A$, $b\in B$, and $s,t\in G$.
Therefore, the composition
 $$(A \otimes_{\max} B)\rtimes_{\gamma} (G \times G)\stackrel{\Pi_{2}}{\longrightarrow}
 (A \otimes_{\max} B)\rtimes_{\alpha \otimes \beta} (G \times G)\stackrel{\Pi}{\longrightarrow}
(A \rtimes_{\alpha} G,i) \otimes_{\max} (B\rtimes_{\beta} G,j)$$
of isomorphisms gives an isomorphism
$$\Pi_{3}:((A \otimes_{\max} B)\rtimes_{\gamma} (G \times G),k)\rightarrow
(A \rtimes_{\alpha} G,i) \otimes_{\max} (B\rtimes_{\beta} G,j)$$
such that
$$\Pi_{3}\big(k_{(A\otimes_{\max} B)}(a \otimes b)k_{G \times G}(s,t)\big)=i_{A}(a)i_{G}(t)\otimes j_{B}(b)j_{G}(s)$$
for all $a\in A$, $b\in B$, and $s,t\in G$.
\end{remark}

\begin{remark}
\label{I+J CP}
Let $(A \rtimes_{\alpha} G,i)$ be the group crossed product of a dynamical system $(A,G,\alpha)$
by discrete group. If $I$ and $J$ are two $\alpha$-invariant ideals of the algebra $A$, then one can compute on spanning elements
to see that we have
\begin{align}
\label{I+J CP2}
(I \rtimes_{\alpha} G)+(J \rtimes_{\alpha} G)=(I+J) \rtimes_{\alpha} G
\end{align}
and
\begin{align}
\label{IJ CP2}
(I \rtimes_{\alpha} G) \cap (J \rtimes_{\alpha} G)=(I\cap J) \rtimes_{\alpha} G.
\end{align}
\end{remark}

\begin{theorem}
\label{J crossed Z2}
Consider the (essential) ideals $\J_{\delta}$, $\J_{\gamma}$ and $\J$ of the algebra $\B$ in the group dynamical system
$(\B,\Z^{2},\beta)$ induced by the system $(A,\N^{2},\alpha)$. The ideals $\J_{\delta}$ and $\J_{\gamma}$ are $\beta$-invariant, such that
\begin{align}
\label{J*Z2=}
(\J_{\delta}\rtimes_{\beta} \Z^{2})+(\J_{\gamma}\rtimes_{\beta} \Z^{2})=\J\rtimes_{\beta} \Z^{2},
\end{align}
and
\begin{align}
\label{C0*Z2=}
(\J_{\delta}\rtimes_{\beta} \Z^{2})\cap (\J_{\gamma}\rtimes_{\beta} \Z^{2})\simeq\K(\ell^{2}(\Z^{2}))\otimes A
\end{align}
which is an essential ideal of $(\B\rtimes_{\beta} \Z^{2},j)$. In addition, the ideals
$\J_{\delta}\rtimes_{\beta} \Z^{2}$ and $\J_{\gamma}\rtimes_{\beta} \Z^{2}$ of $\B\rtimes_{\beta} \Z^{2}$ are essentials, and
isomorphic to the algebras $\K(\ell^{2}(\Z)) \otimes (D\rtimes_{\tau} \Z)$ and $\K(\ell^{2}(\Z)) \otimes (C\rtimes_{\tau} \Z)$
of compact operators, respectively, where $\tau$ denotes the action of $\Z$ on the subalgebras $D$ and $C$ of $\ell^{\infty}(\Z,A)$
by the left translation.

\end{theorem}

\begin{proof}
Firstly, one can easily see that the essential ideals $\J_{\delta}$ and $\J_{\gamma}$ of $\B$ are $\beta$-invariant,
and hence, the algebras $\J_{\delta}\rtimes_{\beta} \Z^{2}$  and $\J_{\gamma}\rtimes_{\beta} \Z^{2}$ are essential ideals of
$\B \rtimes_{\beta} \Z^{2}$ (see \cite[Proposition 2.4]{Kusuda}). Now, the equation (\ref{J*Z2=}) follows immediately by
(\ref{I+J CP2}) in Remark \ref{I+J CP} as $\J_{\delta}+\J_{\gamma}=\J$ (see Theorem \ref{J for Z2}).
It thus follows that $\J\rtimes_{\beta} \Z^{2}$ is indeed spanned by the elements of the form
$$\big[j_{\B}\big(\phi_{(m,n)}(a)-\phi_{(m+1,n)}(\gamma(a))\big)j_{\Z^{2}}(x,y)\big]+
\big[j_{\B}\big(\phi_{(t,u)}(b)-\phi_{(t,u+1)}(\delta(b))\big)j_{\Z^{2}}(r,s)\big],$$
where $a,b\in A$ and $(m,n), (x,y), (t,u), (r,s)\in \Z^{2}$.
To see (\ref{C0*Z2=}), we apply (\ref{IJ CP2}) in Remark \ref{I+J CP}, and since $\J_{\delta}\cap \J_{\gamma}=C_{0}(\Z^{2})\otimes A$
(see Theorem \ref{J for Z2}), we have
$$(\J_{\delta}\rtimes_{\beta} \Z^{2})\cap (\J_{\gamma}\rtimes_{\beta} \Z^{2})=(\J_{\delta}\cap \J_{\gamma})\rtimes_{\beta} \Z^{2}
=(C_{0}(\Z^{2})\otimes A)\rtimes_{\beta} \Z^{2},$$
where it is known that the crossed product $(C_{0}(\Z^{2})\otimes A)\rtimes_{\beta} \Z^{2}$ is isomorphic to the algebra
$$\K(\ell^{2}(\Z^{2})\otimes A)\simeq \K(\ell^{2}(\Z^{2}))\otimes A \simeq \K(\ell^{2}(\Z))\otimes \K(\ell^{2}(\Z))\otimes A$$
of compact operators. Moreover, since $C_{0}(\Z^{2})\otimes A$ is an essential ideal of $\B$,
$(C_{0}(\Z^{2})\otimes A)\rtimes_{\beta} \Z^{2}\simeq \K(\ell^{2}(\Z^{2}))\otimes A$ is an essential ideal of
$\B\rtimes_{\beta} \Z^{2}$ (again by \cite[Proposition 2.4]{Kusuda}).

Next, let $\lt$ denote the action of $\Z$ on $C_{0}(\Z)$ by the left translation. Then, consider the action
$\lt \otimes \tau$ of $\Z^{2}$ on the algebra $C_{0}(\Z)\otimes D$, and let $\sigma$ be the action of $\Z^{2}$ on the algebra
$C_{0}(\Z)\otimes C$ defined by
$$\sigma_{(m,n)}=(\lt \otimes \tau)_{(n,m)}=\lt_{n} \otimes \tau_{m}$$
for all $(m,n)\in \Z^{2}$ (see remark \ref{tensor CP}).
Now, we have
$$(\lt \otimes \tau)\circ \Lambda_{\delta}=\Lambda_{\delta} \circ \beta
\ \textrm{and}\ \sigma\circ \Lambda_{\gamma}=\Lambda_{\gamma} \circ \beta,$$
from which, it follows that the systems $(\J_{\delta}, \Z^{2}, \beta)$ and $(C_{0}(\Z)\otimes D, \Z^{2}, \lt \otimes \tau)$
are equivariantly isomorphic as well as $(\J_{\gamma}, \Z^{2}, \beta)$ and $(C_{0}(\Z)\otimes C, \Z^{2}, \sigma)$.
Therefore, by \cite[Lemma 2.65]{W}, the crossed products $\J_{\delta}\rtimes_{\beta} \Z^{2}$ and $\J_{\gamma}\rtimes_{\beta} \Z^{2}$
are isomorphic to $(C_{0}(\Z) \otimes D)\rtimes_{\lt \otimes \tau} \Z^{2}$ and $(C_{0}(\Z)\otimes C)\rtimes_{\sigma} \Z^{2}$, respectively.
To be more precise, there are isomorphisms
$$\Psi_{1}:(\J_{\delta}\rtimes_{\beta} \Z^{2},j)\rightarrow ((C_{0}(\Z) \otimes D)\rtimes_{\lt \otimes \tau} \Z^{2},k)$$
and
$$\Psi_{2}:(\J_{\gamma}\rtimes_{\beta} \Z^{2},j)\rightarrow ((C_{0}(\Z) \otimes C)\rtimes_{\sigma} \Z^{2},i)$$
such that
$$\Psi_{1}(j_{\B}(\xi)j_{\Z^{2}}(x,y))=k_{(C_{0}(\Z) \otimes D)}(\Lambda_{\delta}(\xi))k_{\Z^{2}}(x,y)$$
and
$$\Psi_{2}(j_{\B}(\eta)j_{\Z^{2}}(x,y))=i_{(C_{0}(\Z) \otimes C)}(\Lambda_{\gamma}(\eta))i_{\Z^{2}}(x,y)$$
for all $\xi\in \J_{\delta}$, $\eta\in \J_{\gamma}$, and $(x,y)\in \Z^{2}$.
Moreover, the crossed products $(C_{0}(\Z) \otimes D)\rtimes_{\lt \otimes \tau} \Z^{2}$ and $(C_{0}(\Z) \otimes C)\rtimes_{\sigma} \Z^{2}$
are decomposed as the tensor products $(C_{0}(\Z) \rtimes_{\lt} \Z) \otimes (D\rtimes_{\tau} \Z)$ and
$(C_{0}(\Z) \rtimes_{\lt} \Z) \otimes (C\rtimes_{\tau} \Z)$, respectively (see Remark \ref{tensor CP}), via the isomorphisms
$$\Psi_{3}:(C_{0}(\Z) \otimes D)\rtimes_{\lt \otimes \tau} \Z^{2}
\rightarrow (C_{0}(\Z) \rtimes_{\lt} \Z,i) \otimes (D\rtimes_{\tau} \Z,k)$$
and
$$\Psi_{4}:(C_{0}(\Z) \otimes C)\rtimes_{\sigma} \Z^{2}
\rightarrow (C_{0}(\Z) \rtimes_{\lt} \Z,i) \otimes (C\rtimes_{\tau} \Z,\tilde{i})$$
such that
$$\Psi_{3}\big(k_{(C_{0}(\Z) \otimes D)}(f\otimes \xi)k_{\Z^{2}}(x,y)\big)
=[i_{C_{0}(\Z)}(f)i_{\Z}(x)]\otimes [k_{D}(\xi)k_{\Z}(y)]$$
and
$$\Psi_{4}\big(i_{(C_{0}(\Z) \otimes C)}(f\otimes \eta)i_{\Z^{2}}(x,y)\big)
=[i_{C_{0}(\Z)}(f)i_{\Z}(y)]\otimes [\tilde{i}_{C}(\eta)\tilde{i}_{\Z}(x)]$$
for all $f\in C_{0}(\Z)$, $\xi\in D$, $\eta\in C$, and $x,y\in \Z$. Also, the crossed product $C_{0}(\Z) \rtimes_{\lt} \Z$
is isomorphic to the algebra $\K(\ell^{2}(\Z))$ of
compact operators on $\ell^{2}(\Z)$. The isomorphism is given by
$$\Psi_{5}:(C_{0}(\Z) \rtimes_{\lt} \Z,i)\rightarrow \K(\ell^{2}(\Z))$$ such that
$$\Psi_{5}(i_{C_{0}(\Z)}(f)i_{\Z}(x))=e_{r}\otimes \overline{e_{r-x}},$$ where $f=(...,0,0,1,0,0,...)\in C_{0}(\Z)$ with $1$
in $r$th slot, $\{e_{r}: r\in \Z\}$ is the usual orthonormal basis of $\ell^{2}(\Z)$, and $e_{r}\otimes \overline{e_{r-x}}$
is the rank-one operator on $\ell^{2}(\Z)$ defined by $h\mapsto \langle h\ |\ e_{r-x}\rangle e_{r}$ for all
$h\in \ell^{2}(\Z)$. Therefore, the composition
$$\J_{\delta}\rtimes_{\beta} \Z^{2}\stackrel{\Psi_{1}}{\longrightarrow}
(C_{0}(\Z) \otimes D)\rtimes_{\lt \otimes \tau} \Z^{2}\stackrel{\Psi_{3}}{\longrightarrow}
(C_{0}(\Z) \rtimes_{\lt} \Z) \otimes (D\rtimes_{\tau} \Z)\stackrel{\Psi_{5}\otimes \id}{\longrightarrow}
\K(\ell^{2}(\Z)) \otimes (D\rtimes_{\tau} \Z)$$
of isomorphisms gives an isomorphism
$$\Psi_{7}:\J_{\delta}\rtimes_{\beta} \Z^{2}\rightarrow \K(\ell^{2}(\Z)) \otimes (D\rtimes_{\tau} \Z)$$
such that
$$\Psi_{7}\big[ j_{\B}\big(\phi_{(m,n)}(a)-\phi_{(m+1,n)}(\gamma(a))\big)j_{\Z^{2}}(x,y) \big]
=\big[e_{m}\otimes \overline{e_{m-x}}\big] \otimes \big[k_{D}(\varphi_{n}(a)) k_{\Z}(y)\big]$$
for all $a\in A$ and $(m,n),(x,y)\in\Z^{2}$. Note that, the restriction of $\Psi_{7}$ to the
ideal $(C_{0}(\Z^{2})\otimes A)\rtimes_{\beta} \Z^{2}$ of $\J_{\delta}\rtimes_{\beta} \Z^{2}$ is the canonical isomorphism of the
crossed product $(C_{0}(\Z^{2})\otimes A)\rtimes_{\beta} \Z^{2}$ onto the algebra $\K(\ell^{2}(\Z))\otimes (\K(\ell^{2}(\Z))\otimes A)$
of compact operators.

Similarly, the composition $(\Psi_{5}\otimes \id)\circ \Psi_{4}\circ \Psi_{2}$ of isomorphisms gives an isomorphism
$$\Psi_{6}:\J_{\gamma}\rtimes_{\beta} \Z^{2}\rightarrow \K(\ell^{2}(\Z)) \otimes (C\rtimes_{\tau} \Z)$$
such that
$$\Psi_{6}\big[ j_{\B}\big(\phi_{(t,u)}(a)-\phi_{(t,u+1)}(\delta(a))\big)j_{\Z^{2}}(r,s) \big]
=\big[e_{u}\otimes \overline{e_{u-s}}\big] \otimes \big[\tilde{i}_{C}(\psi_{t}(a))\tilde{i}_{\Z}(r)\big]$$
for all $a\in A$ and $(t,u),(r,s)\in\Z^{2}$.

\end{proof}

We are now ready to import the information from $\B\rtimes_{\beta} \Z^{2}$ to the crossed product $A\times_{\alpha}^{\piso} \N^{2}$
in order to ge the desired composition series (\ref{compose-1}) of ideals and identify the subquotients with familiar algebras.

\begin{lemma}
\label{I=Id+Ig}
Let
$$\I_{\delta}=\clsp\big\{\xi_{(m,n)}^{(x,y)}(a): m,n,x,y\in \N, a\in A\big\}$$
and
$$\I_{\gamma}=\clsp\big\{\eta_{(m,n)}^{(x,y)}(a): m,n,x,y\in \N, a\in A\big\},$$
where
$$\xi_{(m,n)}^{(x,y)}(a):=i_{\N^{2}}(m,n)^{*} i_{A}(a) [1-i_{\N^{2}}(1,0)^{*}i_{\N^{2}}(1,0)]i_{\N^{2}}(x,y)$$
and
$$\eta_{(m,n)}^{(x,y)}(a):=i_{\N^{2}}(m,n)^{*} i_{A}(a) [1-i_{\N^{2}}(0,1)^{*}i_{\N^{2}}(0,1)]i_{\N^{2}}(x,y).$$
Then, $\I_{\delta}$ and $\I_{\gamma}$ are essential ideals of $A\times_{\alpha}^{\piso} \N^{2}$ such that
$\I_{\delta}+ \I_{\gamma}=\ker q$. Moreover, we have
\begin{align}
\label{xi-xi=eta-eta}
\xi_{(m,n)}^{(x,y)}(a)-\xi_{(m,n+1)}^{(x,y+1)}(\delta(a))=\eta_{(m,n)}^{(x,y)}(a)-\eta_{(m+1,n)}^{(x+1,y)}(\gamma(a))
\end{align}
for all $a\in A$ and $m,n,x,y\in \N$, and
\begin{align}
\label{Id.Ig-span}
\I_{\delta}\cap \I_{\gamma}=\clsp\big\{\xi_{(m,n)}^{(x,y)}(a)-\xi_{(m,n+1)}^{(x,y+1)}(\delta(a)): m,n,x,y\in \N, a\in A\big\}
\end{align}
which is also an essential ideal of $A\times_{\alpha}^{\piso} \N^{2}$.
\end{lemma}

\begin{proof}
Firstly, it follows from Theorem \ref{J crossed Z2} that
\begin{align}
\label{eq25}
p(\J\rtimes_{\beta} \Z^{2})p=p(\J_{\delta}\rtimes_{\beta} \Z^{2}+\J_{\gamma}\rtimes_{\beta} \Z^{2})p
=p(\J_{\delta}\rtimes_{\beta} \Z^{2})p+p(\J_{\gamma}\rtimes_{\beta} \Z^{2})p,
\end{align}
where $p=\overline{j_{\B}}(\overline{\phi}_{(0,0)}(1))$ (see \S\ref{sec:pre}). Then, by exactly the same computation done in
\cite[Lemma 4.2]{SZ2}, one can see that the corners $p(\J_{\delta}\rtimes_{\beta} \Z^{2})p$ and $p(\J_{\gamma}\rtimes_{\beta} \Z^{2})p$
are full. Moreover, they are indeed essential ideals of the algebra (full corner) $p(\B\times_{\beta} \Z^{2})p$.
Now, by applying the isomorphism $\Psi$ of $A\times_{\alpha}^{\piso} \N^{2}$ onto $p(\B\rtimes_{\beta} \Z^{2})p$
(see \cite[Theorem 4.1]{SZ2} or \S\ref{sec:pre}), it follows that
$$\I_{\delta}:=\Psi^{-1}(p(\J_{\delta}\rtimes_{\beta} \Z^{2})p)\ \ \textrm{and}\ \
\I_{\gamma}:=\Psi^{-1}(p(\J_{\gamma}\rtimes_{\beta} \Z^{2})p)$$
are two essential ideals of $A\times_{\alpha}^{\piso} \N^{2}$, such that
$$\ker q=\Psi^{-1}(p(\J\rtimes_{\beta} \Z^{2})p)=\Psi^{-1}(p(\J_{\delta}\rtimes_{\beta} \Z^{2})p)+
\Psi^{-1}(p(\J_{\gamma}\rtimes_{\beta} \Z^{2})p)=\I_{\delta}+ \I_{\gamma}.$$

Next, we show that the ideal $\I_{\delta}$ is spanned by the elements $\xi_{(m,n)}^{(x,y)}(a)$ and skip the discussion
on $\I_{\gamma}$ as it follows similarly. To do so, first note that the algebra $p(\J_{\delta}\rtimes_{\beta} \Z^{2})p$
is spanned by elements of the form
$$p\big[j_{\B}\big(\phi_{(m,n)}(a)-\phi_{(m+1,n)}(\gamma(a))\big)j_{\Z^{2}}(x,y)\big]p,$$ where $a\in A$ and
$(m,n),(x,y)\in \Z^{2}$. However, again by exactly the same computation available in \cite[Lemma 4.2]{SZ2}
along with applying the covariance equation of the pair $(j_{\B},j_{\Z^{2}})$, it follows that
$p(\J_{\delta}\rtimes_{\beta} \Z^{2})p$ is precisely spanned by the elements
\begin{align}
\label{eq24}
p\big[j_{\Z^{2}}(m,n)j_{\B}\big(\phi_{(0,0)}(a)-\phi_{(1,0)}(\gamma(a))\big)j_{\Z^{2}}(x,y)^{*}\big]p,
\end{align}
where $a\in A$ and $(m,n),(x,y)\in \N^{2}$. Therefore, elements of the form
\begin{eqnarray*}
\begin{array}{l}
\Psi^{-1}\big(p[j_{\Z^{2}}(m,n)j_{\B}\big(\phi_{(0,0)}(a)-\phi_{(1,0)}(\gamma(a))\big)j_{\Z^{2}}(x,y)^{*}]p\big)\\
=i_{\N^{2}}(m,n)^{*} i_{A}(a) [1-i_{\N^{2}}(1,0)^{*}i_{\N^{2}}(1,0)]i_{\N^{2}}(x,y)
=\xi_{(m,n)}^{(x,y)}(a)\ (\textrm{see \cite[Lemma 4.2]{SZ2}})
\end{array}
\end{eqnarray*}
span the ideal $\Psi^{-1}(p(\J_{\delta}\rtimes_{\beta} \Z^{2})p)=\I_{\delta}$. Consequently, the ideal $\ker q$ is indeed spanned by the elements
\begin{align}
\label{span-I(2)}
\xi_{(m,n)}^{(x,y)}(a)+\eta_{(r,s)}^{(t,u)}(b).
\end{align}

To see the equation (\ref{xi-xi=eta-eta}), for convenience, let
$$P_{(1,0)}=1-i_{\N^{2}}(1,0)^{*}i_{\N^{2}}(1,0)\ \ \textrm{and}\ \ P_{(0,1)}=1-i_{\N^{2}}(0,1)^{*}i_{\N^{2}}(0,1),$$
which are two projections in $\M(A\times_{\alpha}^{\piso} \N^{2})$.
Now, by applying the covariance equation of the pair $(i_{A},i_{\N^{2}})$, we have
\begin{eqnarray*}
\begin{array}{l}
\xi_{(m,n)}^{(x,y)}(a)-\xi_{(m,n+1)}^{(x,y+1)}(\delta(a))\\
=i_{\N^{2}}(m,n)^{*} i_{A}(a)P_{(1,0)}i_{\N^{2}}(x,y)-i_{\N^{2}}(m,n+1)^{*} i_{A}(\delta(a))P_{(1,0)}i_{\N^{2}}(x,y+1)\\
=i_{\N^{2}}(m,n)^{*} i_{A}(a)P_{(1,0)}i_{\N^{2}}(x,y)-i_{\N^{2}}(m,n)^{*}i_{\N^{2}}(0,1)^{*} i_{A}(\delta(a))P_{(1,0)}i_{\N^{2}}(0,1)i_{\N^{2}}(x,y)\\
=i_{\N^{2}}(m,n)^{*} i_{A}(a)P_{(1,0)}i_{\N^{2}}(x,y)-i_{\N^{2}}(m,n)^{*}i_{A}(a)i_{\N^{2}}(0,1)^{*}P_{(1,0)}i_{\N^{2}}(0,1)i_{\N^{2}}(x,y)\\
=i_{\N^{2}}(m,n)^{*} i_{A}(a)\big[P_{(1,0)}-i_{\N^{2}}(0,1)^{*}P_{(1,0)}i_{\N^{2}}(0,1)\big]i_{\N^{2}}(x,y).\\
\end{array}
\end{eqnarray*}
Then, in the bottom line, for $\big[P_{(1,0)}-i_{\N^{2}}(0,1)^{*}P_{(1,0)}i_{\N^{2}}(0,1)\big]$, we have
\begin{eqnarray*}
\begin{array}{l}
P_{(1,0)}-i_{\N^{2}}(0,1)^{*}P_{(1,0)}i_{\N^{2}}(0,1)\\
=[1-i_{\N^{2}}(1,0)^{*}i_{\N^{2}}(1,0)]-i_{\N^{2}}(0,1)^{*}[1-i_{\N^{2}}(1,0)^{*}i_{\N^{2}}(1,0)]i_{\N^{2}}(0,1)\\
=1-i_{\N^{2}}(1,0)^{*}i_{\N^{2}}(1,0)-i_{\N^{2}}(0,1)^{*}i_{\N^{2}}(0,1)+i_{\N^{2}}(1,1)^{*}i_{\N^{2}}(1,1)\\
=[1-i_{\N^{2}}(0,1)^{*}i_{\N^{2}}(0,1)]-[i_{\N^{2}}(1,0)^{*}i_{\N^{2}}(1,0)-i_{\N^{2}}(1,1)^{*}i_{\N^{2}}(1,1)]\\
=[1-i_{\N^{2}}(0,1)^{*}i_{\N^{2}}(0,1)]-i_{\N^{2}}(1,0)^{*}[1-i_{\N^{2}}(0,1)^{*}i_{\N^{2}}(0,1)]i_{\N^{2}}(1,0)\\
=P_{(0,1)}-i_{\N^{2}}(1,0)^{*}P_{(0,1)}i_{\N^{2}}(1,0).\\
\end{array}
\end{eqnarray*}
Therefore, it follows that
\begin{eqnarray*}
\begin{array}{l}
\xi_{(m,n)}^{(x,y)}(a)-\xi_{(m,n+1)}^{(x,y+1)}(\delta(a))\\
=i_{\N^{2}}(m,n)^{*} i_{A}(a)\big[P_{(0,1)}-i_{\N^{2}}(1,0)^{*}P_{(0,1)}i_{\N^{2}}(1,0)\big]i_{\N^{2}}(x,y)\\
=i_{\N^{2}}(m,n)^{*} i_{A}(a)P_{(0,1)}i_{\N^{2}}(x,y)-i_{\N^{2}}(m,n)^{*} i_{A}(a)i_{\N^{2}}(1,0)^{*}P_{(0,1)}i_{\N^{2}}(1,0)i_{\N^{2}}(x,y)\\
=i_{\N^{2}}(m,n)^{*} i_{A}(a)P_{(0,1)}i_{\N^{2}}(x,y)-i_{\N^{2}}(m,n)^{*}i_{\N^{2}}(1,0)^{*}i_{A}(\gamma(a))P_{(0,1)}i_{\N^{2}}(x+1,y)\\
=i_{\N^{2}}(m,n)^{*} i_{A}(a)P_{(0,1)}i_{\N^{2}}(x,y)-i_{\N^{2}}(m+1,n)^{*}i_{A}(\gamma(a))P_{(0,1)}i_{\N^{2}}(x+1,y)\\
=\eta_{(m,n)}^{(x,y)}(a)-\eta_{(m+1,n)}^{(x+1,y)}(\gamma(a)).
\end{array}
\end{eqnarray*}

Finally, to see (\ref{Id.Ig-span}), it follows immediately by (\ref{xi-xi=eta-eta}) that the right hand side of (\ref{Id.Ig-span})
is contained in $\I_{\delta}\cap \I_{\gamma}$. To see the other inclusion, since
$\I_{\delta}\cap \I_{\gamma}=\I_{\delta} \I_{\gamma}$, it is enough to see that each product
$$\xi_{(m,n)}^{(x,y)}(a)\eta_{(r,s)}^{(t,u)}(b)$$
of the spanning elements of $\I_{\delta}$ and $\I_{\gamma}$ is in the right hand side of (\ref{Id.Ig-span}).
For convenience, first, let
$$\tilde{\xi}_{(m,n)}^{(x,y)}(a):=p[j_{\Z^{2}}(m,n)j_{\B}\big(\phi_{(0,0)}(a)-\phi_{(1,0)}(\gamma(a))\big)j_{\Z^{2}}(x,y)^{*}]p$$
and
$$\tilde{\eta}_{(r,s)}^{(t,u)}(b):=p[j_{\Z^{2}}(r,s)j_{\B}\big(\phi_{(0,0)}(b)-\phi_{(0,1)}(\delta(b))\big)j_{\Z^{2}}(t,u)^{*}]p.$$
Then, since
\begin{align}
\label{eq42}
\xi_{(m,n)}^{(x,y)}(a)\eta_{(r,s)}^{(t,u)}(b)=\Psi^{-1}\big( \tilde{\xi}_{(m,n)}^{(x,y)}(a) \tilde{\eta}_{(r,s)}^{(t,u)}(b) \big),
\end{align}
we need to compute the product
$$\tilde{\xi}_{(m,n)}^{(x,y)}(a) \tilde{\eta}_{(r,s)}^{(t,u)}(b),$$
for which, we can use the equation (\ref{span-Jd.Jg}) in the proof of Theorem \ref{J for Z2}. So, by applying the covariance equation
of $(j_{\B},j_{\Z^{2}})$ and (\ref{suprem2}), we have
\begin{align}
\label{eq40}
\tilde{\xi}_{(m,n)}^{(x,y)}(a) \tilde{\eta}_{(r,s)}^{(t,u)}(b)=
p[j_{\Z^{2}}(m,n)j_{\Z^{2}}(x,y)^{*}j_{\B}(\mu)j_{\Z^{2}}(r,s)j_{\Z^{2}}(t,u)^{*}]p,
\end{align}
where
\begin{align}
\label{eq41}
\mu=\big[ \phi_{(x,y)}(a\overline{\alpha}_{(x,y)}(1))-\phi_{(x+1,y)}(\gamma(a)\overline{\alpha}_{(x+1,y)}(1)) \big]
\big[ \phi_{(r,s)}(b)-\phi_{(r,s+1)}(\delta(b)) \big].
\end{align}
Now, as it was discussed in the proof of Theorem \ref{J for Z2}, if $x\geq r$ and $y\leq s$, then $\mu$ is nonzero. Otherwise, it is
zero, and therefore, the product (\ref{eq42}) becomes zero, which is clearly in the right hand side of (\ref{Id.Ig-span}).
But, if $x\geq r$ and $y\leq s$, by applying the equation (\ref{span-Jd.Jg}), we get
$$\mu=\big[ \phi_{(x,s)}(c)-\phi_{(x,s+1)}(\delta(c)) \big]-\big[ \phi_{(x+1,s)}(\gamma(c))-\phi_{(x+1,s+1)}(\alpha_{(1,1)}(c)) \big],$$
where $c=\delta_{s-y}(a\overline{\alpha}_{(x,y)}(1)) \gamma_{x-r}(b)$. Hence, for (\ref{eq40}), we have
\begin{eqnarray*}
\begin{array}{l}
\tilde{\xi}_{(m,n)}^{(x,y)}(a) \tilde{\eta}_{(r,s)}^{(t,u)}(b)\\
=p\big[ j_{\Z^{2}}(m,n)j_{\Z^{2}}(x,y)^{*}j_{\B}\big( \phi_{(x,s)}(c)-\phi_{(x,s+1)}(\delta(c)) \big)j_{\Z^{2}}(r,s)j_{\Z^{2}}(t,u)^{*} \big]p\\
-p\big[ j_{\Z^{2}}(m,n)j_{\Z^{2}}(x,y)^{*}j_{\B}\big( \phi_{(x+1,s)}(\gamma(c))-\phi_{(x+1,s+1)}(\alpha_{(1,1)}(c)) \big)
j_{\Z^{2}}(r,s)j_{\Z^{2}}(t,u)^{*} \big]p,
\end{array}
\end{eqnarray*}
where again by applying the covariance equation of $(j_{\B},j_{\Z^{2}})$, it follows that
\begin{eqnarray*}
\begin{array}{l}
\tilde{\xi}_{(m,n)}^{(x,y)}(a) \tilde{\eta}_{(r,s)}^{(t,u)}(b)\\
=p\big[ j_{\Z^{2}}(m-x,n-y)
j_{\Z^{2}}(x,s) j_{\B}\big( \phi_{(0,0)}(c)-\phi_{(0,1)}(\delta(c)) \big) j_{\Z^{2}}(x,s)^{*} j_{\Z^{2}}(t-r,u-s)^{*} \big]p\\
-p\big[ j_{\Z^{2}}(m-x,n-y)
j_{\Z^{2}}(x+1,s) j_{\B}\big( \phi_{(0,0)}(\gamma(c))-\phi_{(0,1)}(\alpha_{(1,1)}(c)) \big)\\
j_{\Z^{2}}(x+1,s)^{*} j_{\Z^{2}}(t-r,u-s)^{*}  \big]p\\
=p\big[ j_{\Z^{2}}(m,n+s-y) j_{\B}\big( \phi_{(0,0)}(c)-\phi_{(0,1)}(\delta(c)) \big) j_{\Z^{2}}(x+t-r,u)^{*} \big]p\\
-p\big[ j_{\Z^{2}}(m+1,n+s-y) j_{\B}\big( \phi_{(0,0)}(\gamma(c))-\phi_{(0,1)}(\delta(\gamma(c))) \big)j_{\Z^{2}}(x+t-r+1,u)^{*} \big]p\\
=p\big[ j_{\Z^{2}}(m,k) j_{\B}\big( \phi_{(0,0)}(c)-\phi_{(0,1)}(\delta(c)) \big) j_{\Z^{2}}(i,u)^{*} \big]p\\
-p\big[ j_{\Z^{2}}(m+1,k) j_{\B}\big( \phi_{(0,0)}(\gamma(c))-\phi_{(0,1)}(\delta(\gamma(c))) \big)j_{\Z^{2}}(i+1,u)^{*} \big]p\\
=\tilde{\eta}_{(m,k)}^{(i,u)}(c)-\tilde{\eta}_{(m+1,k)}^{(i+1,u)}(\gamma(c)),
\end{array}
\end{eqnarray*}
where $k=n+s-y$ and $i=x+t-r$, which belong to $\N$. Consequently, for the product (\ref{eq42}), we get
\begin{eqnarray*}
\begin{array}{rcl}
\xi_{(m,n)}^{(x,y)}(a)\eta_{(r,s)}^{(t,u)}(b)&=&\Psi^{-1}\big( \tilde{\xi}_{(m,n)}^{(x,y)}(a) \tilde{\eta}_{(r,s)}^{(t,u)}(b) \big)\\
&=&\Psi^{-1}\big( \tilde{\eta}_{(m,k)}^{(i,u)}(c)-\tilde{\eta}_{(m+1,k)}^{(i+1,u)}(\gamma(c)) \big)\\
&=&\eta_{(m,k)}^{(i,u)}(c)-\eta_{(m+1,k)}^{(i+1,u)}(\gamma(c))\\
&=&\xi_{(m,k)}^{(i,u)}(c)-\xi_{(m,k+1)}^{(i,u+1)}(\delta(c))\ \ (\textrm{by}\ (\ref{xi-xi=eta-eta}))
\end{array}
\end{eqnarray*}
which belongs to the right hand side of (\ref{Id.Ig-span}). Thus, (\ref{Id.Ig-span}) is valid.
Note that $\I_{\delta}\cap \I_{\gamma}$ is an essential ideal of $A\times_{\alpha}^{\piso} \N^{2}$ as $\I_{\delta}$ and $\I_{\gamma}$
both are. This completes the proof.
\end{proof}

\begin{prop}
\label{idnfy-ideals}
Suppose that $(A\times_{\delta}^{\piso} \N,j_{A},v)$ and $(A\times_{\gamma}^{\piso} \N,\iota_{A},w)$
are the partial-isometric crossed products of the dynamical systems $(A,\N,\delta)$ and $(A,\N,\gamma)$, respectively.
Then, the (essential) ideals $\I_{\delta}$ and $\I_{\gamma}$ of $A\times_{\alpha}^{\piso} \N^{2}$ are full corners in
algebras $\K(\ell^{2}(\N)) \otimes (A\times_{\delta}^{\piso} \N)$ and $\K(\ell^{2}(\N)) \otimes (A\times_{\gamma}^{\piso} \N)$
of compact operators, respectively.
\end{prop}

\begin{proof}
We only provide the proof on the ideal $\I_{\delta}$ as the proof on $\I_{\gamma}$ follows similarly.
Firstly, by Lemma \ref{I=Id+Ig},
$$\I_{\delta}\stackrel{\Psi}\simeq p(\J_{\delta}\rtimes_{\beta} \Z^{2})p,$$
where $p=\overline{j_{\B}}(\overline{\phi}_{(0,0)}(1))$. Let the map
$$\Pi:\K(\ell^{2}(\Z))\otimes (D\rtimes_{\tau}\Z)\rightarrow \K\big(\ell^{2}(\Z) \otimes (D\rtimes_{\tau} \Z)\big)$$
be the canonical isomorphism such that
$$\Pi((e_{m}\otimes \overline{e_{n}})\otimes \xi \eta^{*})=\Theta_{e_{m}\otimes \xi, e_{n}\otimes \eta}$$
for all $m,n\in\Z$ and $\xi,\eta \in (D\rtimes_{\tau} \Z,k)$, where $\{e_{m}: m\in\Z\}$ is the usual orthonormal basis of $\ell^{2}(\Z)$.
Then, $\Omega:=(\Pi\circ\Psi_{7})$ (see the isomorphism $\Psi_{7}$ in Theorem \ref{J crossed Z2}) is an isomorphism
of $\J_{\delta}\rtimes_{\beta} \Z^{2}$ onto $\K\big(\ell^{2}(\Z) \otimes (D\rtimes_{\tau} \Z)\big)$, such that it maps each
spanning element $$j_{\B}\big(\phi_{(m,n)}(ab^{*})-\phi_{(m+1,n)}(\gamma(ab^{*}))\big)j_{\Z^{2}}(x,y)$$
of $\J_{\delta}\rtimes_{\beta} \Z^{2}$ to the (spanning) element
$$\Theta_{[e_{m}\otimes k_{D}(\varphi_{n}(a))],[e_{m-x}\otimes k_{\Z}(-y)k_{D}(\varphi_{n}(b))]},$$
where $a,b\in A$. Therefore, we have
\begin{align}
\label{eq45}
\I_{\delta}\stackrel{\Psi}\simeq p(\J_{\delta}\rtimes_{\beta} \Z^{2})p \stackrel{\Omega}\simeq
\overline{\Omega}(p) \K\big(\ell^{2}(\Z) \otimes (D\rtimes_{\tau} \Z)\big) \overline{\Omega}(p),
\end{align}
where $\overline{\Omega}(p)$ is a projection in $\M\big(\K\big(\ell^{2}(\Z) \otimes (D\rtimes_{\tau} \Z)\big)\big)\simeq
 \L\big(\ell^{2}(\Z) \otimes (D\rtimes_{\tau} \Z)\big)$ which we denote it by $P_{\delta}$.
We claim that
\begin{align}\label{P-delta}
(P_{\delta}f)(m)=
   \begin{cases}
      \overline{k_{D}\circ\varphi_{0}}(\overline{\gamma}_{m}(1))f(m) &\textrm{if}\empty\ \text{$m\geq 0,$}\\
      0 &\textrm{if}\empty\ \text{$m< 0$}
   \end{cases}
\end{align}
for all $f\in \ell^{2}(\Z) \otimes (D\rtimes_{\tau} \Z)$. To prove our claim, it suffices to see that
\[
P_{\delta}\big(e_{m}\otimes k_{D}(\varphi_{n}(ab^{*}))\xi\big)=
   \begin{cases}
      e_{m}\otimes [\overline{k_{D}\circ\varphi_{0}}(\overline{\gamma}_{m}(1))]k_{D}(\varphi_{n}(ab^{*}))\xi
      &\textrm{if}\empty\ \text{$m\geq 0$,}\\
      0 &\textrm{if}\empty\ \text{$m< 0$}
   \end{cases}
\]
on the spanning element $[e_{m}\otimes k_{D}(\varphi_{n}(ab^{*}))\xi]$ of
$\ell^{2}(\Z) \otimes (D\rtimes_{\tau} \Z)$, where $\xi\in D\rtimes_{\tau} \Z$. Since
\begin{eqnarray*}
\begin{array}{rcl}
e_{m}\otimes k_{D}(\varphi_{n}(ab^{*}))\xi&=&
\Theta_{[e_{m}\otimes k_{D}(\varphi_{n}(a))],
[e_{m}\otimes k_{D}(\varphi_{n}(b))]}(e_{m}\otimes \xi)\\
&=&\Omega\bigg(j_{\B}\big(\phi_{(m,n)}(ab^{*})-\phi_{(m+1,n)}(\gamma(ab^{*}))\big)\bigg)(e_{m}\otimes \xi),
\end{array}
\end{eqnarray*}
it follows that
\begin{eqnarray}
\label{eq43}
\begin{array}{rcl}
P_{\delta}\big(e_{m}\otimes k_{D}(\varphi_{n}(ab^{*}))\xi\big)&=&
\overline{\Omega}(p) \Omega\bigg(j_{\B}\big(\phi_{(m,n)}(ab^{*})-\phi_{(m+1,n)}(\gamma(ab^{*}))\big)\bigg)(e_{m}\otimes \xi)\\
&=&\Omega\bigg(p j_{\B}\big(\phi_{(m,n)}(ab^{*})-\phi_{(m+1,n)}(\gamma(ab^{*}))\big)\bigg)(e_{m}\otimes \xi).\\
\end{array}
\end{eqnarray}
Now, in the bottom line, as
\begin{eqnarray*}
\begin{array}{rcl}
pj_{\B}\big(\phi_{(m,n)}(ab^{*})-\phi_{(m+1,n)}(\gamma(ab^{*}))\big)&=&
\overline{j_{\B}}(\overline{\phi}_{(0,0)}(1))j_{\B}\big(\phi_{(m,n)}(ab^{*})-\phi_{(m+1,n)}(\gamma(ab^{*}))\big)\\
&=&j_{\B}\big(\overline{\phi}_{(0,0)}(1)[\phi_{(m,n)}(ab^{*})-\phi_{(m+1,n)}(\gamma(ab^{*}))]\big),
\end{array}
\end{eqnarray*}
we need to compute the product
\begin{align}
\label{eq31}
\overline{\phi}_{(0,0)}(1)[\phi_{(m,n)}(ab^{*})-\phi_{(m+1,n)}(\gamma(ab^{*}))].
\end{align}
To do so, we consider two cases $m\geq 0$ and $m<0$ separately. If $m\geq 0$, calculation by applying (\ref{suprem2}) shows that
\begin{align}
\label{eq44}
\overline{\phi}_{(0,0)}(1)[\phi_{(m,n)}(ab^{*})-\phi_{(m+1,n)}(\gamma(ab^{*}))]
=\phi_{(m,t)}(cd^{*})-\phi_{(m+1,t)}(\gamma(cd^{*})),
\end{align}
where $c=\overline{\alpha}_{(m,t)}(1)\delta_{t-n}(a)$, $d^{*}=\delta_{t-n}(b^{*})$, and $t=0\vee n$.
Therefore, by applying (\ref{eq44}) to (\ref{eq43}), we get
\begin{eqnarray*}
\begin{array}{rcl}
P_{\delta}\big(e_{m}\otimes k_{D}(\varphi_{n}(ab^{*}))\xi\big)&=&
\Omega\bigg( j_{\B}\big(\phi_{(m,t)}(cd^{*})-\phi_{(m+1,t)}(\gamma(cd^{*}))\big)\bigg)(e_{m}\otimes \xi)\\
&=&\Theta_{[e_{m}\otimes k_{D}(\varphi_{t}(c))],[e_{m}\otimes k_{D}(\varphi_{t}(d))]}(e_{m}\otimes \xi)
=e_{m}\otimes k_{D}(\varphi_{t}(cd^{*}))\xi.
\end{array}
\end{eqnarray*}
Moreover, in the bottom line, for $\varphi_{t}(cd^{*})$, we have
\begin{eqnarray*}
\begin{array}{rcl}
\varphi_{t}(cd^{*})&=&\varphi_{t}(\overline{\alpha}_{(m,t)}(1)\delta_{t-n}(ab^{*}))\\
&=&\varphi_{t}(\overline{\delta}_{t}(\overline{\gamma}_{m}(1))\delta_{t-n}(ab^{*}))\\
&=&\overline{\varphi}_{0}(\overline{\gamma}_{m}(1))\varphi_{n}(ab^{*})\ \ (\textrm{recall that}\ t=0\vee n).
\end{array}
\end{eqnarray*}
Consequently,
\begin{eqnarray*}
\begin{array}{rcl}
P_{\delta}\big(e_{m}\otimes k_{D}(\varphi_{n}(ab^{*}))\xi\big)&=&
e_{m}\otimes k_{D}(\overline{\varphi}_{0}(\overline{\gamma}_{m}(1))\varphi_{n}(ab^{*}))\xi\\
&=&e_{m}\otimes \overline{k_{D}}(\overline{\varphi}_{0}(\overline{\gamma}_{m}(1)))k_{D}(\varphi_{n}(ab^{*}))\xi\\
&=&e_{m}\otimes [\overline{k_{D}\circ\varphi_{0}}(\overline{\gamma}_{m}(1))]k_{D}(\varphi_{n}(ab^{*}))\xi.
\end{array}
\end{eqnarray*}
If $m<0$, again by applying (\ref{suprem2}), one can calculate to see that the product (\ref{eq31}) equals zero, and hence,
for (\ref{eq43}), it follows that
$$P_{\delta}\big(e_{m}\otimes k_{D}(\varphi_{n}(ab^{*}))\xi\big)=\Omega(j_{\B}(0))(e_{m}\otimes \xi)=0.$$
Thus, (\ref{P-delta}) is indeed valid. Next, we show that
$P_{\delta}\K\big(\ell^{2}(\Z) \otimes (D\rtimes_{\tau} \Z)\big)P_{\delta}$ is actually equal to the corner
$$Q_{\delta}\K\big(\ell^{2}(\N) \otimes (A\times_{\delta}^{\piso} \N)\big)Q_{\delta}$$ of the algebra
$\K\big(\ell^{2}(\N) \otimes (A\times_{\delta}^{\piso} \N)\big)\simeq \K(\ell^{2}(\N)) \otimes (A\times_{\delta}^{\piso} \N)$,
where $Q_{\delta}$ is a projection in
$\M\big(\K\big(\ell^{2}(\N) \otimes (A\times_{\delta}^{\piso} \N)\big)\big) \simeq
\L\big(\ell^{2}(\N) \otimes (A\times_{\delta}^{\piso} \N)\big)$
defined by
$$(Q_{\delta}f)(m)=\overline{j_{A}}(\overline{\gamma}_{m}(1))f(m)$$
for all $f\in \ell^{2}(\N) \otimes (A\times_{\delta}^{\piso} \N)$. To do so,
recall that first, by \cite[Theorem 4.1]{SZ} (or \cite[Theorem 4.1]{SZ2}), the crossed product $A\times_{\delta}^{\piso} \N$
is isomorphic to the full corner $r(D\rtimes_{\tau}\Z)r$, where $r$ is the projection
$\overline{k_{D}\circ\varphi_{0}}(1)$ in the multiplier algebra $\M(D\rtimes_{\tau}\Z)$.
Then, since $\K\big(\ell^{2}(\Z) \otimes (D\rtimes_{\tau} \Z)\big)$ is spanned by the elements (compact operators)
$\big\{\Theta_{e_{m}\otimes \xi, e_{n}\otimes \eta}: m,n\in\Z,\ \xi,\eta \in (D\rtimes_{\tau} \Z)\big\}$, we have
\begin{eqnarray*}
\begin{array}{rcl}
P_{\delta}\K\big(\ell^{2}(\Z) \otimes (D\rtimes_{\tau} \Z)\big)P_{\delta}&=&
\clsp\{P_{\delta}(\Theta_{e_{m}\otimes \xi, e_{n}\otimes \eta})P_{\delta}: m,n\in\Z,\ \xi,\eta \in (D\rtimes_{\tau} \Z)\}\\
&=&\clsp\{\Theta_{P_{\delta}(e_{m}\otimes \xi), P_{\delta}(e_{n}\otimes \eta)}: m,n\in\Z,\ \xi,\eta \in (D\rtimes_{\tau} \Z)\}.
\end{array}
\end{eqnarray*}
However, if $m<0$ or $n<0$, then $P_{\delta}(e_{m}\otimes \xi)=0$ or $P_{\delta}(e_{n}\otimes \eta)=0$, and hence,
$\Theta_{P_{\delta}(e_{m}\otimes \xi), P_{\delta}(e_{n}\otimes \eta)}=0$. It thus follows that
$$P_{\delta}\K\big(\ell^{2}(\Z) \otimes (D\rtimes_{\tau} \Z)\big)P_{\delta}=
\clsp\{\Theta_{P_{\delta}(e_{m}\otimes \xi), P_{\delta}(e_{n}\otimes \eta)}: m,n\in\N,\ \xi,\eta \in (D\rtimes_{\tau} \Z)\}.$$
Moreover, since
\begin{eqnarray*}
\begin{array}{rcl}
P_{\delta}(e_{m}\otimes \xi)&=&e_{m}\otimes \overline{k_{D}\circ\varphi_{0}}(\overline{\gamma}_{m}(1))\xi\\
&=&e_{m}\otimes \overline{k_{D}\circ\varphi_{0}}(\overline{\gamma}_{m}(1))\overline{k_{D}\circ\varphi_{0}}(1)\xi\\
&=&e_{m}\otimes \overline{k_{D}\circ\varphi_{0}}(\overline{\gamma}_{m}(1))r\xi
=P_{\delta}(e_{m}\otimes r\xi),
\end{array}
\end{eqnarray*}
we have
\begin{align}
\label{eq33}
\Theta_{P_{\delta}(e_{m}\otimes \xi), P_{\delta}(e_{n}\otimes \eta)}
=\Theta_{P_{\delta}(e_{m}\otimes r \xi), P_{\delta}(e_{n}\otimes r \eta)}
=P_{\delta}(\Theta_{e_{m}\otimes r \xi, e_{n}\otimes r \eta})P_{\delta}
\end{align}
for all $m,n\in\N$ and $\xi,\eta \in D\rtimes_{\tau} \Z$. However, each element
$\Theta_{[e_{m}\otimes r \xi], [e_{n}\otimes r\eta]}$ is actually a compact operator in
$\K\big(\ell^{2}(\N)\otimes (A\times_{\delta}^{\piso} \N)\big)$.
This is due to the facts that
$$\Theta_{[e_{m}\otimes r \xi], [e_{n}\otimes r\eta]}=\Pi((e_{m}\otimes \overline{e_{n}})\otimes r\xi \eta^{*} r),$$
where
$r\xi \eta^{*} r\in r(D\rtimes_{\tau}\Z)r \simeq A\times_{\delta}^{\piso} \N$, and the restriction of the isomorphism $\Pi$ to the subalgebra
$\K(\ell^{2}(\N))\otimes (A\times_{\delta}^{\piso} \N)$ of $\K(\ell^{2}(\Z))\otimes (D\rtimes_{\tau}\Z)$
gives the canonical isomorphism of $\K(\ell^{2}(\N))\otimes (A\times_{\delta}^{\piso} \N)$
onto $\K\big(\ell^{2}(\N)\otimes (A\times_{\delta}^{\piso} \N)\big)$. Eventually, since for every $m\in\N$
and $a\in (A\times_{\delta}^{\piso} \N)$,
\begin{align}
\label{eq34}
P_{\delta}(e_{m}\otimes a)=e_{m}\otimes (\overline{k_{D}\circ\varphi_{0}})(\overline{\gamma}_{m}(1))a
=e_{m}\otimes \overline{j_{A}}(\overline{\gamma}_{m}(1))a=Q_{\delta}(e_{m}\otimes a),
\end{align}
it follows that (see (\ref{eq33}))
\begin{eqnarray*}
\begin{array}{l}
P_{\delta}\K\big(\ell^{2}(\Z) \otimes (D\rtimes_{\tau} \Z)\big)P_{\delta}\\
=\clsp\big\{P_{\delta}(\Theta_{e_{m}\otimes r\xi, e_{n}\otimes r\eta})P_{\delta}:
m,n\in\N, \xi,\eta \in (D\rtimes_{\tau} \Z)\big\}\\
=\clsp\big\{P_{\delta}\Pi\big((e_{m}\otimes \overline{e_{n}})\otimes r\xi \eta^{*}r\big)P_{\delta}:
m,n\in\N, \xi,\eta \in (D\rtimes_{\tau} \Z)\big\}\\
=\clsp\big\{P_{\delta}\Pi\big((e_{m}\otimes \overline{e_{n}})\otimes ab^{*}\big)P_{\delta}:
m,n\in\N, a,b \in (A\times_{\delta}^{\piso} \N)\big\}\\
=\clsp\big\{P_{\delta}(\Theta_{e_{m}\otimes a, e_{n}\otimes b})P_{\delta}:m,n\in\N, a,b \in (A\times_{\delta}^{\piso} \N)\big\}\\
=\clsp\big\{\Theta_{P_{\delta}(e_{m}\otimes a), P_{\delta}(e_{n}\otimes b)}:m,n\in\N, a,b \in (A\times_{\delta}^{\piso} \N)\big\}\\
=\clsp\big\{\Theta_{Q_{\delta}(e_{m}\otimes a), Q_{\delta}(e_{n}\otimes b)}: m,n\in\N, a,b\in (A\times_{\delta}^{\piso} \N)\big\}\\
=\clsp\big\{Q_{\delta}(\Theta_{e_{m}\otimes a, e_{n}\otimes b})Q_{\delta}: m,n\in\N, a,b\in (A\times_{\delta}^{\piso} \N)\big\}\\
=Q_{\delta}\K\big(\ell^{2}(\N)\otimes (A\times_{\delta}^{\piso} \N)\big)Q_{\delta}.
\end{array}
\end{eqnarray*}
Consequently,
$$\I_{\delta}\stackrel{\Psi}\simeq p(\J_{\delta}\rtimes_{\beta} \Z^{2})p \stackrel{\Omega}\simeq
P_{\delta}\K\big(\ell^{2}(\Z) \otimes (D\rtimes_{\tau} \Z)\big)P_{\delta}
=Q_{\delta}\K\big(\ell^{2}(\N)\otimes (A\times_{\delta}^{\piso} \N)\big)Q_{\delta}.$$
More precisely, the composition $\Omega\circ \Psi$ of isomorphisms gives an isomorphism
$$\Psi_{\delta}: \I_{\delta} \rightarrow Q_{\delta}\K\big(\ell^{2}(\N)\otimes (A\times_{\delta}^{\piso} \N)\big)Q_{\delta}$$
such that
\begin{eqnarray}
\label{psi-d}
\begin{array}{l}
\Psi_{\delta}\big( i_{\N^{2}}(m,n)^{*} i_{A}(ab^{*}) [1-i_{\N^{2}}(1,0)^{*}i_{\N^{2}}(1,0)]i_{\N^{2}}(x,y) \big)\\
=Q_{\delta}\Pi\big((e_{m}\otimes \overline{e_{x}}) \otimes v_{n}^{*}j_{A}(ab^{*})v_{y}\big)Q_{\delta}
=Q_{\delta}\big(\Theta_{[e_{m}\otimes v_{n}^{*}j_{A}(a)], [e_{x}\otimes v_{y}^{*}j_{A}(b)]}\big)Q_{\delta}.
\end{array}
\end{eqnarray}
Note that, by applying the covariance equations of the pairs $(j_{\B},j_{\Z^{2}})$ and $(k_{D},k_{\Z})$, (\ref{eq33}), and
(\ref{eq34}), one can calculate on the spanning elements of $\I_{\delta}$ to see (\ref{psi-d}).

To see that $\I_{\delta}$ is a full corner in $\K(\ell^{2}(\N)\otimes (A\times_{\delta}^{\piso} \N))$, first note that the
algebra $\K(\ell^{2}(\N)\otimes (A\times_{\delta}^{\piso} \N))$ is spanned by the elements
\begin{align}
\label{eq36}
\Theta_{[e_{m}\otimes v_{n}^{*}j_{A}(a)v_{r}], [e_{x}\otimes v_{y}^{*}j_{A}(bc^{*})v_{s}]}.
\end{align}
Now, if $\{a_{\lambda}\}$ is an approximate unit in $A$, then the spanning element (\ref{eq36}) is the norm-limit of the net
\begin{align}
\label{eq37}
\Theta_{[e_{m}\otimes v_{n}^{*}j_{A}(aa_{\lambda})v_{r}], [e_{x}\otimes v_{y}^{*}j_{A}(bc^{*})v_{s}]}
\end{align}
in $\K(\ell^{2}(\N)\otimes (A\times_{\delta}^{\piso} \N))$. But, calculation shows that, for (\ref{eq37}), we have
\begin{eqnarray*}
\begin{array}{l}
\Theta_{[e_{m}\otimes v_{n}^{*}j_{A}(aa_{\lambda})v_{r}], [e_{x}\otimes v_{y}^{*}j_{A}(bc^{*})v_{s}]}\\
=\big(\Theta_{[e_{m}\otimes v_{n}^{*}j_{A}(a)], [e_{0}\otimes v_{r}^{*}j_{A}(a_{\lambda})]}\big)
\big(\Theta_{[e_{0}\otimes v_{s}^{*}j_{A}(c)], [e_{x}\otimes v_{y}^{*}j_{A}(b)]}\big)\\
=\big(\Theta_{[e_{m}\otimes v_{n}^{*}j_{A}(a)], [e_{0}\otimes v_{r}^{*}j_{A}(a_{\lambda})]}\big)Q_{\delta}
\big(\Theta_{[e_{0}\otimes v_{s}^{*}j_{A}(c)], [e_{x}\otimes v_{y}^{*}j_{A}(b)]}\big),\\
\end{array}
\end{eqnarray*}
which belongs to
\begin{align}
\label{K Q-del K}
\overline{\K(\ell^{2}(\N)\otimes (A\times_{\delta}^{\piso} \N))Q_{\delta}\K(\ell^{2}(\N)\otimes (A\times_{\delta}^{\piso} \N))}.
\end{align}
It therefore follows that each spanning element (\ref{eq36}) of $\K(\ell^{2}(\N)\otimes (A\times_{\delta}^{\piso} \N))$ must belong to
(\ref{K Q-del K}), and hence, we have
$$\K(\ell^{2}(\N)\otimes (A\times_{\delta}^{\piso} \N))
=\overline{\K(\ell^{2}(\N)\otimes (A\times_{\delta}^{\piso} \N))Q_{\delta}\K(\ell^{2}(\N)\otimes (A\times_{\delta}^{\piso} \N))}.$$

At last, it follows by a similar discussion that there is an isomorphism $\Psi_{\gamma}$ of the ideal $\I_{\gamma}$ onto the full
corner $Q_{\gamma}\K(\ell^{2}(\N)\otimes (A\times_{\gamma}^{\piso} \N))Q_{\gamma}$, where $Q_{\gamma}$ is a projection in
$\M\big(\K\big(\ell^{2}(\N) \otimes (A\times_{\gamma}^{\piso} \N)\big)\big)\simeq \L\big(\ell^{2}(\N) \otimes (A\times_{\gamma}^{\piso} \N)\big)$
defined by
$$(Q_{\gamma}h)(n)=\overline{\iota_{A}}(\overline{\delta}_{n}(1))h(n)$$
for all $h\in \ell^{2}(\N) \otimes (A\times_{\gamma}^{\piso} \N)$. The isomorphism $\Psi_{\gamma}$ maps each spanning element
$$i_{\N^{2}}(r,s)^{*} i_{A}(ab^{*}) [1-i_{\N^{2}}(0,1)^{*}i_{\N^{2}}(0,1)]i_{\N^{2}}(t,u)$$
of $\I_{\gamma}$ to the (spanning) element
$$Q_{\gamma}\big(\Theta_{[e_{s}\otimes w_{r}^{*}\iota_{A}(a)], [e_{u}\otimes w_{t}^{*}\iota_{A}(b)]}\big)Q_{\gamma}.$$
This completes the proof.
\end{proof}

\begin{theorem}
\label{main-TH}
Let $A\times_{\alpha}^{\piso} \N^{2}$ be the partial-isometric crossed product of the system $(A,\N^{2},\alpha)$ in which
the action $\alpha$ on $A$ is given by extendible endomorphisms. Then, there is a composition series
\begin{align}
\label{main-compose}
0\leq L_{1}\leq L_{2} \leq A\times_{\alpha}^{\piso} \N^{2}
\end{align}
of essential ideals, such that:
\begin{itemize}
\item[(i)] the ideal $L_{1}$ is (isomorphic to) a full corner in the algebra $\K(\ell^{2}(\N^{2}))\otimes A$ of compact operators,
\item[(ii)] $L_{2}/L_{1}\simeq \A_{\delta} \oplus \A_{\gamma}$, and
\item[(iii)] $(A\times_{\alpha}^{\piso} \N^{2})/L_{2}\simeq A\times_{\alpha}^{\iso} \N^{2}$,
\end{itemize}
where the algebras $\A_{\delta}$ and $\A_{\gamma}$ are full corners in algebras
$\K(\ell^{2}(\N)) \otimes (A\times_{\delta}^{\iso} \N)$ and $\K(\ell^{2}(\N)) \otimes (A\times_{\gamma}^{\iso} \N)$
of compact operators, respectively.
\end{theorem}

\begin{proof}
The composition series (\ref{main-compose}) is
$$0\leq \I_{\delta}\cap \I_{\gamma}\leq \ker q \leq A\times_{\alpha}^{\piso} \N^{2}.$$
So, (iii) is indeed true (see the exact sequence (\ref{exseq1})).

To see (i), first note that we have
\begin{eqnarray*}
\begin{array}{rcl}
\I_{\delta}\cap \I_{\gamma}\simeq \Psi(\I_{\delta}\cap \I_{\gamma})&=&\Psi(\I_{\delta}) \cap \Psi(\I_{\gamma})\\
&=&[p(\J_{\delta}\rtimes_{\beta} \Z^{2})p]\cap [p(\J_{\gamma}\rtimes_{\beta} \Z^{2})p]\\
&=&[p(\J_{\delta}\rtimes_{\beta} \Z^{2})p][p(\J_{\gamma}\rtimes_{\beta} \Z^{2})p].
\end{array}
\end{eqnarray*}
Then, it follows by inspection on spanning elements that (see also Theorem \ref{J crossed Z2})
$$[p(\J_{\delta}\rtimes_{\beta} \Z^{2})p][p(\J_{\gamma}\rtimes_{\beta} \Z^{2})p]=p[(C_{0}(\Z^{2})\otimes A)\rtimes_{\beta} \Z^{2}]p,$$
and therefore,
$$\I_{\delta}\cap \I_{\gamma}\stackrel{\Psi}\simeq p[(C_{0}(\Z^{2})\otimes A)\rtimes_{\beta} \Z^{2}]p.$$
Now, if $(\rho,U)$ is the covariant representation of the system $(\B,\Z^{2},\beta)$ in $\L(\ell^{2}(\Z^{2})\otimes A)$
mentioned in section \ref{sec:pre}, the restriction of the corresponding (non-degenerate) representation
$\rho\times U$ of $\B\rtimes_{\beta} \Z^{2}$ to $(C_{0}(\Z^{2})\otimes A)\rtimes_{\beta} \Z^{2}$ is the canonical isomorphism
of $(C_{0}(\Z^{2})\otimes A)\rtimes_{\beta} \Z^{2}$ onto the algebra
$\K(\ell^{2}(\Z^{2})\otimes A)\simeq \K(\ell^{2}(\Z^{2}))\otimes A$ of compact operators. So, we have
\begin{align}
\label{eq46}
\I_{\delta}\cap \I_{\gamma}\stackrel{\Psi}\simeq p[(C_{0}(\Z^{2})\otimes A)\rtimes_{\beta} \Z^{2}]p
\stackrel{(\rho\times U)}\simeq R\K(\ell^{2}(\Z^{2})\otimes A)R,
\end{align}
where $R=\overline{\rho\times U}(p)$ is a projection in
$\M\big(\K\big(\ell^{2}(\Z^{2})\otimes A\big)\big)\simeq\L\big(\ell^{2}(\Z^{2})\otimes A\big)$. By taking an approximate unit
$\{a_{\lambda}\}$ in $A$ and using the equation
$$(\rho\times U)(j_{\B}(\phi_{(0,0)}(a_{\lambda})))=\rho(\phi_{(0,0)}(a_{\lambda})),$$
one can see that $R=\overline{\rho}(\overline{\phi}_{(0,0)}(1))$. Moreover, calculation on spanning elements of
$\ell^{2}(\Z^{2})\otimes A$ shows that we have
\[
(Rf)(m,n)=
   \begin{cases}
      \overline{\alpha}_{(m,n)}(1)f(m,n) &\textrm{if}\empty\ \text{$(m,n)\in \N^{2}$,}\\
      0 &\textrm{otherwise}
   \end{cases}
\]
for all $f\in (\ell^{2}(\Z^{2})\otimes A)$. Each $Rf$ is obviously in $\ell^{2}(\N^{2})\otimes A$. Now, define a map
$Q:\ell^{2}(\N^{2})\otimes A\rightarrow \ell^{2}(\N^{2})\otimes A$
by
$$(Qh)(m,n)=\overline{\alpha}_{(m,n)}(1)h(m,n)\ \ \textrm{for all}\ h\in(\ell^{2}(\N^{2})\otimes A).$$
One can see that $Q$ is a projection in
$\M\big(\K\big(\ell^{2}(\N^{2})\otimes A\big)\big)\simeq \L\big(\ell^{2}(\N^{2})\otimes A\big)$.
Then, by a similar calculation done in Proposition \ref{idnfy-ideals}, we get
$$R\K(\ell^{2}(\Z^{2})\otimes A)R=Q\K(\ell^{2}(\N^{2})\otimes A)Q,$$
and consequently,
$$\I_{\delta}\cap \I_{\gamma}\stackrel{\Psi}\simeq p[(C_{0}(\Z^{2})\otimes A)\rtimes_{\beta} \Z^{2}]p
\stackrel{(\rho\times U)}\simeq R\K(\ell^{2}(\Z^{2})\otimes A)R=Q\K(\ell^{2}(\N^{2})\otimes A)Q.$$
So, the composition $(\rho\times U)\circ\Psi$ gives an
isomorphism $$\phi:\I_{\delta}\cap \I_{\gamma}\rightarrow Q\K(\ell^{2}(\N^{2})\otimes A)Q$$
such that
$$\phi\bigg(\xi_{(m,n)}^{(x,y)}(ab^{*})-\xi_{(m,n+1)}^{(x,y+1)}(\delta(ab^{*}))\bigg)=
Q\big(\Theta_{[e_{(m,n)}\otimes a],[e_{(x,y)}\otimes b]}\big)Q$$
for all $a,b\in A$ and $m,n,x,y\in \N$. To see that $L_{1}=\I_{\delta}\cap \I_{\gamma}$ is a full corner, first note that
$\K(\ell^{2}(\N^{2})\otimes A)$ is spanned by the elements
of the form $\Theta_{[e_{(m,n)}\otimes ab],[e_{(r,s)}\otimes c]}$. Now, if $\{a_{\lambda}\}$ is any approximate unit in $A$, then
\begin{align}
\label{eq27}
\Theta_{[e_{(m,n)}\otimes aa_{\lambda}b],[e_{(r,s)}\otimes c]}\rightarrow \Theta_{[e_{(m,n)}\otimes ab],[e_{(r,s)}\otimes c]}
\end{align}
in the norm topology. But, calculation shows that
\begin{eqnarray}\label{eq28}
\begin{array}{rcl}
\Theta_{[e_{(m,n)}\otimes aa_{\lambda}b],[e_{(r,s)}\otimes c]}&=&
\big(\Theta_{[e_{(m,n)}\otimes a],[e_{(0,0)}\otimes a_{\lambda}]}\big) \big(\Theta_{[e_{(0,0)}\otimes b],[e_{(r,s)}\otimes c]}\big)\\
&=&\big(\Theta_{[e_{(m,n)}\otimes a],[e_{(0,0)}\otimes a_{\lambda}]}\big) Q \big(\Theta_{[e_{(0,0)}\otimes b],[e_{(r,s)}\otimes c]}\big).
\end{array}
\end{eqnarray}
Thus, it follows from (\ref{eq27}) and (\ref{eq28}) that
$$\Theta_{[e_{(m,n)}\otimes ab],[e_{(r,s)}\otimes c]}\in \overline{\K(\ell^{2}(\N^{2})\otimes A)Q\K(\ell^{2}(\N^{2})\otimes A)},$$
and hence, we have
$$\overline{\K(\ell^{2}(\N^{2})\otimes A)Q\K(\ell^{2}(\N^{2})\otimes A)}=\K(\ell^{2}(\N^{2})\otimes A).$$

To see (ii), firstly,
\begin{align}
\label{L2/L1}
L_{2}/L_{1}=(\I_{\delta}+\I_{\gamma})/(\I_{\delta}\cap \I_{\gamma})
\simeq\I_{\delta}/(\I_{\delta}\cap \I_{\gamma}) \oplus \I_{\gamma}/(\I_{\delta}\cap \I_{\gamma}).
\end{align}

Next, we show that the quotients in the above direct sum are isomorphic to full corners in algebras
$\K(\ell^{2}(\N)) \otimes (A\times_{\delta}^{\iso} \N)$ and $\K(\ell^{2}(\N)) \otimes (A\times_{\gamma}^{\iso} \N)$, respectively.
We only do this for
\begin{align}
\label{d-quotient}
\I_{\delta}/(\I_{\delta}\cap \I_{\gamma})
\end{align}
as the other one follows similarly. Firstly, by Proposition \ref{idnfy-ideals},
$$\I_{\delta}\stackrel{\Psi_{\delta}}\simeq Q_{\delta}\K\big(\ell^{2}(\N)\otimes (A\times_{\delta}^{\piso} \N)\big)Q_{\delta}.$$
Let $I$ be the kernel of the natural surjective homomorphism $\varphi$ of $(A\times_{\delta}^{\piso} \N, j_{A}, v)$ onto
the isometric crossed product $(A\times_{\delta}^{\iso} \N, k_{A}, u)$ of the system $(A,\N,\delta)$, where
$$\varphi(v_{m}^{*}j_{A}(a)v_{n})=u_{m}^{*}k_{A}(a)u_{n}\ \ \ \textrm{for all}\ a\in A,\ m,n\in \N.$$
See in \cite{AZ} that $I$ is spanned by the elements $\{v_{m}^{*}j_{A}(a)(1-v^{*}v)v_{n}:a\in A, m,n\in\N\}$, which is an essential
ideal of $A\times_{\delta}^{\piso} \N$. Moreover, it is a full corner in algebra $\K(\ell^{2}(\N))\otimes A$.
Now, calculation on the spanning elements of the ideal $\I_{\delta}\cap \I_{\gamma}$ shows that
\begin{eqnarray*}
\begin{array}{rcl}
\Psi_{\delta}\bigg(\xi_{(m,n)}^{(x,y)}(ab^{*})-\xi_{(m,n+1)}^{(x,y+1)}(\delta(ab^{*}))\bigg)
&=&Q_{\delta}\Pi\big((e_{m}\otimes \overline{e_{x}})\otimes v_{n}^{*}j_{A}(ab^{*})(1-v^{*}v)v_{y}\big)Q_{\delta}\\
&=&Q_{\delta}\big(\Theta_{[e_{m}\otimes v_{n}^{*}j_{A}(a)(1-v^{*}v)],[e_{x}\otimes v_{y}^{*}j_{A}(b)(1-v^{*}v)]}\big)Q_{\delta},
\end{array}
\end{eqnarray*}
which implies that
$$\I_{\delta}\cap \I_{\gamma}\simeq \Psi_{\delta}(\I_{\delta}\cap \I_{\gamma})=Q_{\delta}\K(\ell^{2}(\N)\otimes I)Q_{\delta}.$$
Then, consider the following diagram
\begin{equation*}
\begin{diagram}\dgARROWLENGTH=0.4\dgARROWLENGTH
\node{\K\big(\ell^{2}(\N)\otimes (A\times_{\delta}^{\piso} \N)\big)} \arrow{s,l}{}\arrow{e}
\arrow{s}\arrow{e}\node{\K\big(\ell^{2}(\N)\otimes (A\times_{\delta}^{\iso} \N)\big)}\\
\node{\K(\ell^{2}(\N))\otimes (A\times_{\delta}^{\piso} \N)} \arrow{e,t}{\id\otimes \varphi}
\node{\K(\ell^{2}(\N))\otimes (A\times_{\delta}^{\iso} \N),} \arrow{n,r}{}
\end{diagram}
\end{equation*}
where the vertical arrows denote the canonical isomorphisms. It induces a surjective homomorphism
$$\widetilde{\varphi}:\K\big(\ell^{2}(\N)\otimes (A\times_{\delta}^{\piso} \N)\big)\rightarrow
\K\big(\ell^{2}(\N)\otimes (A\times_{\delta}^{\iso} \N)\big)$$
such that
$$\widetilde{\varphi}(\Theta_{e_{m}\otimes \xi, e_{n}\otimes \eta})=\Theta_{e_{m}\otimes \varphi(\xi), e_{n}\otimes \varphi(\eta)}$$
for all $\xi,\eta\in A\times_{\delta}^{\piso} \N$ and $m,n\in \N$. One can see that indeed
$\ker \widetilde{\varphi}=\K(\ell^{2}(\N)\otimes I)$.
Now, the restriction of $\widetilde{\varphi}$ to the (full) corner
$Q_{\delta}\K\big(\ell^{2}(\N)\otimes (A\times_{\delta}^{\piso} \N)\big)Q_{\delta}$ gives a
surjective homomorphism $\widetilde{\varphi}|$ of it onto the subalgebra
\begin{align}
\label{eq38}
\widetilde{\varphi}|\big(Q_{\delta}\K\big(\ell^{2}(\N)\otimes (A\times_{\delta}^{\piso} \N)\big)Q_{\delta}\big)
=R_{\delta} \K\big(\ell^{2}(\N)\otimes (A\times_{\delta}^{\iso} \N)\big) R_{\delta}
\end{align}
of $\K\big(\ell^{2}(\N)\otimes (A\times_{\delta}^{\iso} \N)\big)$, where $R_{\delta}=\overline{\widetilde{\varphi}}(Q_{\delta})$
is a projection in
$$\M\big(\K\big(\ell^{2}(\N)\otimes (A\times_{\delta}^{\iso} \N)\big)\big)\simeq\L\big(\ell^{2}(\N)\otimes (A\times_{\delta}^{\iso} \N)\big).$$
By calculating on the spanning elements of $\ell^{2}(\N)\otimes (A\times_{\delta}^{\iso} \N)$, it follows that
$$(R_{\delta}f)(m)=\overline{k_{A}}(\overline{\gamma}_{m}(1))f(m)\ \ \textrm{for all}\ f\in \ell^{2}(\N) \otimes (A\times_{\delta}^{\iso} \N).$$
Let us denote the corner $R_{\delta} \K\big(\ell^{2}(\N)\otimes (A\times_{\delta}^{\iso} \N)\big) R_{\delta}$ by $\A_{\delta}$.
Since $\I_{\delta}$ is a full corner in $\K\big(\ell^{2}(\N)\otimes (A\times_{\delta}^{\piso} \N)\big)$, one can apply the
surjection $\widetilde{\varphi}$ to see that $\A_{\delta}$ is actually a full corner in
$\K\big(\ell^{2}(\N)\otimes (A\times_{\delta}^{\iso} \N)\big)$. Now, it follows from
\begin{eqnarray*}
\begin{array}{rcl}
\ker \widetilde{\varphi}|
&=&\ker \widetilde{\varphi}\cap \big[Q_{\delta}\K\big(\ell^{2}(\N)\otimes (A\times_{\delta}^{\piso} \N)\big)Q_{\delta}\big]\\
&=&\K(\ell^{2}(\N)\otimes I)\cap \big[Q_{\delta}\K\big(\ell^{2}(\N)\otimes (A\times_{\delta}^{\piso} \N)\big)Q_{\delta}\big]
\end{array}
\end{eqnarray*}
that
$\ker \widetilde{\varphi}|=Q_{\delta}\K(\ell^{2}(\N)\otimes I)Q_{\delta}$.
Consequently, we get
\begin{eqnarray}
\label{Id/QKQ}
\begin{array}{rcl}
\I_{\delta}/(\I_{\delta}\cap \I_{\gamma})&\simeq&
[Q_{\delta}\K\big(\ell^{2}(\N)\otimes (A\times_{\delta}^{\piso} \N)\big)Q_{\delta}]/[Q_{\delta}\K(\ell^{2}(\N)\otimes I)Q_{\delta}]\\
&\simeq& R_{\delta}\K\big(\ell^{2}(\N)\otimes (A\times_{\delta}^{\iso} \N)\big)R_{\delta}=\A_{\delta}.
\end{array}
\end{eqnarray}
Similarly, there is a surjective homomorphism of $Q_{\gamma}\K\big(\ell^{2}(\N)\otimes (A\times_{\gamma}^{\piso} \N)\big)Q_{\gamma}$
onto the full corner $\A_{\gamma}=R_{\gamma}\K\big(\ell^{2}(\N)\otimes (A\times_{\gamma}^{\iso} \N)\big)R_{\gamma}$
whose kernel is
$$Q_{\gamma}\K(\ell^{2}(\N)\otimes J)Q_{\gamma}=\Psi_{\gamma}(\I_{\delta}\cap \I_{\gamma})\simeq \I_{\delta}\cap \I_{\gamma},$$
where $R_{\gamma}$ is a projection in
$\M\big(\K\big(\ell^{2}(\N)\otimes (A\times_{\gamma}^{\iso} \N)\big)\big)=\L(\ell^{2}(\N)\otimes (A\times_{\gamma}^{\iso} \N))$,
and $J$ is the essential of $A\times_{\gamma}^{\piso} \N$ such that $(A\times_{\gamma}^{\piso} \N)/J\simeq A\times_{\gamma}^{\iso} \N$
(see again \cite{AZ}). Therefore,
\begin{eqnarray}
\label{Ig/QKQ}
\begin{array}{rcl}
\I_{\gamma}/(\I_{\delta}\cap \I_{\gamma})&\simeq&
[Q_{\gamma}\K\big(\ell^{2}(\N)\otimes (A\times_{\gamma}^{\piso} \N)\big)Q_{\gamma}]/[Q_{\gamma}\K(\ell^{2}(\N)\otimes J)Q_{\gamma}]\\
&\simeq& R_{\gamma}\K\big(\ell^{2}(\N)\otimes (A\times_{\gamma}^{\iso} \N)\big)R_{\gamma}=\A_{\gamma}.
\end{array}
\end{eqnarray}
Consequently, it follows by (\ref{Id/QKQ}) and (\ref{Ig/QKQ}) that (see (\ref{L2/L1}))
$$L_{2}/L_{1}\simeq\I_{\delta}/(\I_{\delta}\cap \I_{\gamma}) \oplus \I_{\gamma}/(\I_{\delta}\cap \I_{\gamma})
\simeq \A_{\delta} \oplus \A_{\gamma}.$$
This completes the proof.

\end{proof}

\begin{remark}
\label{auto for N2}
If in the system $(A,\N^{2},\alpha)$ the action $\alpha$ on $A$ is given by automorphisms, then since
$A\times_{\alpha}^{\iso} \N^{2}\simeq A\rtimes_{\alpha} \Z^{2}$, the short exact sequence (\ref{exseq1}) is
\begin{align}
\label{exseq3}
0 \longrightarrow \ker q \stackrel{}{\longrightarrow} A\times_{\alpha}^{\piso} \N^{2}
\stackrel{q}{\longrightarrow} A\rtimes_{\alpha} \Z^{2} \longrightarrow 0.
\end{align}
Moreover, since the systems $(A,\N,\delta)$ and $(A,\N,\gamma)$ are obviously given by automorphic actions,
the algebras $D$ and $C$ are isomorphic to $B_{\Z}\otimes A$, where $B_{\Z}$ is the subalgebra of $\ell^{\infty}(\Z,A)$
generated by the characteristic functions $\{1_{n}: n\in \Z\}$ (see \cite[Proposition 5.1]{SZ}). Therefore, by
\cite[Corollary 5.2]{SZ} or \cite[Corollary 5.3]{AZ}, the algebras $A\times_{\delta}^{\piso} \N$ and $A\times_{\gamma}^{\piso} \N$
are full corners in the group crossed products
$$(B_{\Z}\otimes A)\rtimes_{\lt\otimes\delta^{-1}} \Z\ \ \textrm{and}\ \ (B_{\Z}\otimes A)\rtimes_{\lt\otimes\gamma^{-1}} \Z,$$
respectively.
\end{remark}

\begin{cor}
\label{idfy-idl-auto}
Let $(A,\N^{2},\alpha)$ be a system in which the action $\alpha$ on $A$ is given by automorphisms. Then, the (essential) ideals
$\I_{\delta}$ and $\I_{\gamma}$ of $A\times_{\alpha}^{\piso} \N^{2}$ are isomorphic to the algebras
$\K(\ell^{2}(\N)) \otimes (A\times_{\delta}^{\piso} \N)$ and $\K(\ell^{2}(\N)) \otimes (A\times_{\gamma}^{\piso} \N)$
of compact operators, respectively.
\end{cor}

\begin{proof}
This is due to the fact that, since each $\delta_{n}$ as well as each $\gamma_{n}$ is an automorphism,
the projections $Q_{\delta}$ and $Q_{\gamma}$ in Proposition \ref{idnfy-ideals} become just identity operators.
\end{proof}

\begin{cor}
\label{main-Cor}
Let $(A,\N^{2},\alpha)$ be a system in which the action $\alpha$ on $A$ is given by automorphisms. Then, there is
a composition series
\begin{align}
\label{compose-auto}
0\leq L_{1}\leq L_{2} \leq A\times_{\alpha}^{\piso} \N^{2}
\end{align}
of essential ideals, such that:
\begin{itemize}
\item[(i)] $L_{1}\simeq \K(\ell^{2}(\N^{2}))\otimes A$,
\item[(ii)] $L_{2}/L_{1}\simeq \big[\K(\ell^{2}(\N)) \otimes (A\rtimes_{\delta} \Z)\big]
\oplus \big[\K(\ell^{2}(\N)) \otimes (A\rtimes_{\gamma} \Z)\big]$, and
\item[(iii)] $(A\times_{\alpha}^{\piso} \N^{2})/L_{2}\simeq A\rtimes_{\alpha} \Z^{2}$.
\end{itemize}
\end{cor}

\begin{proof}
Since each $\alpha_{(m,n)}$ is an automorphism, the projections $Q$, $R_{\delta}$, and $R_{\gamma}$ in Theorem \ref{main-TH}
are just identity operators. Moreover, we have
$A\times_{\alpha}^{\iso} \N^{2}\simeq A\rtimes_{\alpha} \Z^{2}$, $A\times_{\delta}^{\iso} \N\simeq A\rtimes_{\delta} \Z$, and
$A\times_{\gamma}^{\iso} \N\simeq A\rtimes_{\gamma} \Z$. So, the rest follows from Theorem \ref{main-TH}.
\end{proof}

\begin{remark}
\label{ext-seq4}
Note that, for the trivial system $(\C,\N^{2},\id)$, the short exact sequence (\ref{exseq3}), is just the well-known exact sequence
\begin{align}
\label{Murphy-exseq3}
0 \longrightarrow \mathcal{C}_{\Z^{2}} \stackrel{}{\longrightarrow} \T(\Z^{2}) \stackrel{}{\longrightarrow} C(\TT^{2}) \longrightarrow 0,
\end{align}
where $\mathcal{C}_{\Z^{2}}$ is the commutator ideal of the Toeplitz algebra $\T(\Z^{2})\simeq \T(\Z)\otimes \T(\Z)$
(see also the remark prior to \cite[Corollary 5.5]{SZ2}). Moreover, the essential ideals $\I_{\delta}$ and $\I_{\gamma}$ of
$\C\times_{\id}^{\piso} \N^{2}\simeq \T(\Z^{2})$ are both isomorphic to the algebra
$$\K(\ell^{2}(\N)) \otimes (\C\times_{\id}^{\piso} \N)\simeq\K(\ell^{2}(\N)) \otimes \T(\Z),$$
where $\C\times_{\id}^{\piso} \N \simeq \T(\Z)$ is known by \cite[Example 4.3]{AZ} (see also the remark prior to \cite[Lemma 5.4]{SZ2}).
Also, in this case, we have
$$L_{1}=\I_{\delta}\cap \I_{\gamma}\simeq \K(\ell^{2}(\N^{2}))\simeq \K(\ell^{2}(\N))\otimes \K(\ell^{2}(\N)),$$
and since $\C\rtimes_{\id} \Z\simeq C^{*}(\Z)\simeq C(\widehat{\Z})\simeq C(\TT),$
\begin{eqnarray*}
\begin{array}{rcl}
L_{2}/L_{1}\simeq \mathcal{C}_{\Z^{2}}/\K(\ell^{2}(\N^{2}))
&\simeq&[\K(\ell^{2}(\N)) \otimes (\C\rtimes_{\id} \Z)]\oplus [\K(\ell^{2}(\N)) \otimes (\C\rtimes_{\id} \Z)]\\
&\simeq&[\K(\ell^{2}(\N)) \otimes C(\TT)]\oplus [\K(\ell^{2}(\N)) \otimes C(\TT)].
\end{array}
\end{eqnarray*}

\end{remark}

\subsection*{Acknowledgements}
This work (Grant No. RGNS 64-102) was financially supported by Office of the Permanent Secretary, Ministry of Higher Education, Science,
Research and Innovation.

\end{document}